\documentclass[11pt,reqno,draft]{amsart}
\usepackage{graphicx,subfigure}
\usepackage{hyperref}
\usepackage[english]{babel} 
\usepackage[utf8]{inputenc} 
\usepackage[T1]{fontenc} 
\usepackage{amsmath}
\usepackage{amsfonts}
\usepackage{amssymb}
\usepackage{amsthm}
\usepackage{comment}
\usepackage{lmodern}

\numberwithin{equation}{section} 
\numberwithin{table}{section} 

{\theoremstyle{plain}
\newtheorem{theorem}{Theorem}[section]
\newtheorem{proposition}[theorem]{Proposition}
\newtheorem{lemma}[theorem]{Lemma}
\newtheorem{corollary}[theorem]{Corollary}
\newtheorem{remark}[theorem]{Remark}
\newtheorem{example}[theorem]{Example}

\newtheorem{definition}[theorem]{Definition}
}

\def\R{\mathbb{R}}
\def\Z{\mathbb{Z}}

\begin{document}

\title[Concentration driven by an external magnetic field]{Nonlinear Schr\"odinger equation: concentration on circles driven by an external magnetic field}

\author{Denis Bonheure}
\address{Denis Bonheure,
\newline \indent D{\'e}partement de Math{\'e}matique, Universit{\'e} libre de Bruxelles,
\newline \indent CP 214,  Boulevard du Triomphe, B-1050 Bruxelles, Belgium
\newline \indent and INRIA - Team MEPHYSTO.}
\email{denis.bonheure@ulb.ac.be}

\author{Silvia Cingolani}
\address{Silvia Cingolani,
\newline \indent Dipartimento di Meccanica, Matematica e Management, Politecnico di Bari, \newline \indent Via E. Orabona 4, 70125 Bari, Italy.}
\email{silvia.cingolani@poliba.it}

\author{Manon Nys}
\address{Manon Nys \newline \indent Fonds National de la Recherche Scientifique- FNRS.
\newline \indent Département de Mathématique, Université Libre de Bruxelles, 
\newline \indent CP 214, Boulevard du triomphe, B-1050 Bruxelles, Belgium.
\newline \indent Dipartimento di Matematica e Applicazioni,  Universit\`{a} degli Studi di Milano-Bicocca,
\newline \indent via Bicocca degli Arcimboldi 8, 20126 Milano, Italy.}
\email{manonys@ulb.ac.be}

\date{\today}

\thanks{
D.B. is supported by INRIA - Team MEPHYSTO, MIS F.4508.14 (FNRS), PDR T.1110.14F (FNRS) 
\& ARC AUWB-2012-12/17-ULB1- IAPAS. He thanks the support of INDAM during his visits at the Politecnico di Bari where part of this work has been done. 
S.C. is partially supported by GNAMPA-INDAM Project 2015 \emph{Analisi variazionale di modelli fisici non lineari.} She thanks the support of FNRS during her visit at Universit\'e libre de Bruxelles where part of this work has been done.
M.N. is a Research Fellow of the Belgian Fonds de la Recherche Scientifique - FNRS and is partially supported by the project ERC Advanced Grant  2013 n. 339958: ``Complex Patterns for Strongly Interacting Dynamical Systems -COMPAT''. She thanks the Universit\`{a} di Torino for hospitality.
}

\maketitle

\begin{abstract}
In this paper, we study the semiclassical limit for the stationary magnetic nonlinear Schr\"odinger equation 
\begin{align}\label{eq:initialabstract}
\left( i \hbar \nabla + A(x) \right)^2 u + V(x) u = |u|^{p-2} u, \quad x\in \mathbb{R}^{3},
\end{align}
where $p>2$, $A$ is a vector potential associated to a given magnetic field $B$, i.e $\nabla \times A =B$ and  $V$ is a nonnegative, scalar (electric) potential which can be singular at the origin and vanish at infinity or outside a compact set.
We assume that $A$ and $V$ satisfy a cylindrical symmetry. By a refined penalization argument, we prove the existence of semiclassical cylindrically symmetric solutions of $(\ref{eq:initialabstract})$ whose moduli concentrate, as $\hbar \to 0$, around a circle. We emphasize that the concentration is driven by the magnetic and the electric potentials. 
Our result thus shows that in the semiclassical limit, the magnetic field also influences the location of the solutions of  $(\ref{eq:initialabstract})$ if their concentration occurs around a locus, not a single point. 
\\ 
\break
2010 \emph{AMS Subject Classification.}
\\
\emph{Keywords} Nonlinear Schrödinger equation; semiclassical states; Singular potential; Vanishing potential; Concentration on curves; External magnetic field; Variational methods; penalization method.
\end{abstract}

\section{Introduction} \label{section:introduction}

In Quantum Mechanics, the nonlinear Schrödinger equation (NLS) with a exterior magnetic field $B$, having source in the magnetic potential $A$, and a scalar (electric) potential $U$ has the form
\begin{align*} 
i \hbar \frac{\partial \psi}{\partial{t}} = \left(i \hbar \nabla + A(x) \right)^2 \psi +U(x) \psi= f(|\psi|^2) \psi, \qquad x\in \mathbb{R}^{N},
\end{align*}
where $N \geq 3$, $i^2=-1$, $\hbar$ is the Planck constant, and the mass is taken $m=1/2$ for simplicity. The magnetic laplacian is defined by
 \begin{align*}
\left( i \hbar \nabla + A(x) \right)^2 = - \hbar^2 \Delta   + 2 i \hbar A(x) \cdot \nabla + i \hbar \operatorname{div}A(x) + |A(x)|^2 ,
\end{align*} 
and $f(|\psi|^2) \psi$ is a nonlinear term. In dimension $N=3$, the magnetic potential $A$ is related to the magnetic field $B$ by the relation $B = \nabla \times A$. Such evolution equation arises in various physical contexts, such as nonlinear optics or plasma physics, where one simulates the interaction effect among many particles by introducing a nonlinear term. 

The search of standing waves $\psi(x,t)= e^{-i \frac{E}{\hbar} t}  \, u(x)$ leads to study the stationary nonlinear magnetic Schrödinger equation
\begin{align} \label{eq:initialproblem}
	\left( i \hbar \nabla + A(x) \right)^2 u + (U(x)-E) u = f(|u|^{2}) u, \qquad x\in \mathbb{R}^{N}.
\end{align}
In the following, we  write $V(x) = U(x) - E$ and for simplicity we consider $f(t)= t^{(p-2)/2}$. However, we note that a larger class of nonlinearity could be considered, see for instance \cite{Rabinowitz}.

\medbreak
For $\hbar >0$ fixed, the existence of a solution of \eqref{eq:initialproblem} whose modulus $|u_{\hbar}|$ vanishes at infinity was first proved by Esteban and Lions in \cite{EsLi} by using a constrained minimization approach. Concentration and compactness arguments are applied to solve the associated minimization problems for a broad class of magnetic fields. Successively in \cite{ArioliSzulkin}, Arioli and Szulkin studied the existence of infinitely many solutions of \eqref{eq:initialproblem} assuming that $V$ and $B$ are periodic.

\medbreak
In the present paper, we are interested in the semiclassical analysis of the magnetic nonlinear Schr\"odinger equation \eqref{eq:initialproblem}. From a mathematical point of view, the transition from quantum to classical mechanics can be formally performed by letting $\hbar\rightarrow0$. For small values of $\hbar>0$, solutions $u_{\hbar} : \mathbb{R}^{N}\rightarrow\mathbb{C}$ of \eqref{eq:initialproblem} are usually referred to as semiclassical (ground or bound) states.

\medbreak 
When $A=0$, the study of the nonlinear Schrödinger equation 
\begin{align} \label{eq:initialproblem2}
- \hbar^2 \Delta u + V(x) u = f(|u|^{2}) u, \qquad x\in \mathbb{R}^{N},
\end{align}
has been extensively pursued in the semiclassical regime and a considerable amount of work has been done, showing that existence and concentration phenomena of single- and multiple-spike solutions occur at critical points of the electric potential $V$ when $\hbar \rightarrow0$, see e.g. \cite{ABC1997,AmbrosettiMalchiodiBook,CingolaniLazzo1997,CingolaniLazzo2000,DP-F,DelPinoFelmer1997,FW1986,Oh1989,Rabinowitz}. Successively, the question of existence of semiclassical solutions to NLS equations concentrating on higher dimensional sets has been investigated. In \cite{AmbrosettiMalchiodiNi2003} Ambrosetti, Malchiodi and Ni considered the case of a radial potential $V(|x|)$ and constructed radial solutions exhibiting concentration on  a sphere, which radius is a non degenerate critical point of the concentration function $\mathcal{M}(r)= r^{N-1} V^\sigma (r)$, $\sigma =  p/(p-2) - 1/2$ (see also \cite{AmbrosettiRuiz2006,BadialeDaprile2002}). Moreover they conjectured that this phenomenon takes place, at least along a sequence $\hbar_n \to 0$, whenever the sphere is replaced by a closed hypersurface $\Gamma$, stationary and non degenerate for the weighted area functional $\int_{\Gamma}  V^\sigma$. In \cite{DelPinoKowalczykWei2007}, the above conjecture was completely solved in the plane by Del Pino, Kowalczyk and Wei. We also quote \cite{MalchiodiMontenegro2002}, where Malchiodi and Montenegro considered the NLS equation on a smooth bounded domain $\Omega$ in $\R^2$ with Neumann boundary conditions and proved, for a suitable sequence $\hbar_n \rightarrow 0$, the existence of positive solutions $u_{\hbar_n}$ concentrating at the whole boundary of $\Omega$ or at some components of it. In \cite{Malchiodi2005,Malchiodi2004}, boundary concentration on a geodesic of the boundary has been treated in the three-dimensional case. Later on, concentration on spheres of dimension $N-2$ was studied in \cite{MollePassaseo2006} whereas Ambrosetti and Malchiodi proved the existence of solutions concentrating on $k$-dimensional spheres ($k\le 1\le N-1$), see \cite[Theorem 10.11]{AmbrosettiMalchiodiBook}. More recently in \cite{B-DC-VS}, Bonheure, Di Cosmo and Van Schaftingen proved the existence of semiclassical solutions to \eqref{eq:initialproblem2} concentrating on a $k$-dimensional sphere, $1 \leq k \leq N-1$, for a large class of symmetric potentials $V$. 

\medbreak
In presence of a magnetic field ($A \neq 0$), a challenging question is to establish how the magnetic field influences the existence and  the concentration of the moduli of the complex-valued solutions of \eqref{eq:initialproblem} as $\hbar \to 0$. A first result dealing with the concentration of least-energy solutions for magnetic NLS equations was obtained  in \cite{Kurata}. In this paper, Kurata proved that if    $(u_{\hbar})_{\hbar}$ is a sequence of least-energy solutions to \eqref{eq:initialproblem} with $f(t)= t^{(p-2)/2}$, then the sequence  $(|u_{\hbar}|)_{\hbar}$ of their moduli must concentrate at a global minimum $x_0$ of $V$, as $\hbar \to 0$. More precisely, there exist a sequence of points $(x_n)_n \subset \mathbb{R}^N$ and a subsequence still denoted by $(\hbar_n)_n$, with $x_n \rightarrow x_0$ and $\hbar_n \to 0$ as $n \to + \infty$, such that $v_{\hbar_n}(y)= u_{\hbar_n}(x_n + \hbar_n y)$ converges to some $v \in C^2_{\text{loc}}$ and converges also weakly in $L^p$. Moreover, $v$ satisfies the limiting equation
\begin{align*}
\left( i  \nabla + A(x_0) \right)^2 v + V(x_0)v = |v|^{p-2}v, \qquad x \in \R^N.
\end{align*}
If we let $w(x)= e^{- i A(x_0) \cdot x}\, v(x)$, it follows that $w$ satisfies weakly the equation  
\begin{align*}
-\Delta w + V(x_0) w = |w|^{p-2} w, \qquad x \in \R^N. 
\end{align*}
Hence the concentration of the least-energy solutions is driven by the electric potential while the magnetic potential influences the phase factor of the solutions, but does not affect the location of the peaks of their moduli. The existence of such semiclassical least-energy solutions for magnetic NLS equations was established in \cite{Cingolani2003} by using Ljusternick-Schnirelmann theory.

Successively, by using a penalization argument, the existence of semiclassical bound state solutions to \eqref{eq:initialproblem}, concentrating at local minima of $V$, has been proved in \cite{CingolaniSecchi2005}, for a large class of magnetic potentials, covering the case of polynomial growths corresponding to constant magnetic field (see also \cite{CingolaniSecchi2002} for bounded potentials).

We also refer to \cite{CingolaniJeanjeanSecchi2009} for existence results of multi-peak solutions to \eqref{eq:initialproblem}, whose moduli have multiple concentration points around local minima of $V$, dealing with a large class of nonlinear terms (possibly not monotone); and to \cite{CingolaniClapp2009} for semiclassical solutions having specific symmetries concentrating around orbits of critical points of $V$.

In \cite{SecchiSquassina2005}, the authors have established necessary conditions for a sequence of standing wave solutions of \eqref{eq:initialproblem} to concentrate, in different senses, around a given point. More precisely, they show that if $f(t) = t^{(p-2)/2}$, then the moduli of the peaks have to locate at critical points of $V$, independently of $A$, confirming what was conjectured in \cite{CingolaniSecchi2002}.

In all the above cited papers, the concentration of the moduli of the complex valued solutions occurs at one or a finite set of critical points of the electric field $V$, while the magnetic field only influences the phase factor of the standing waves as $\hbar$ is small.

\medbreak

In the present paper, we are interested in studying concentration phenomena on higher dimensional sets in the presence of a magnetic field. More particularly, we aim to understand how and in which situations the magnetic field influences such concentration. In the following, we restrict ourself to consider \eqref{eq:initialproblem} in $\mathbb{R}^3$, for which we can already detect some interesting phenomena.

More specifically, we consider the class of scalar potentials $V$ invariant under a group $G$ of orthogonal transformations, and the class of magnetic potentials $A$ equivariant under the same group, that is 
\begin{align} \label{eq:equivariantA}
g \, A(g^{-1}x) = A(x),
\end{align}
for every $g \in G$.

\smallskip

In dimension $3$, the simplest group is $G = O(3)$ which corresponds to a radially symmetric setting. The potential $V$ then depends only on $|x|$, while $A$ satisfies the equivariance condition \eqref{eq:equivariantA} for every $g \in O(3)$. However, this last constraint on $A$ is too strong, in the sense that the only possible vector potential satisfying this condition is a multiple of the normal vector to the sphere. Indeed, if $x$ is a point on a sphere of radius $r$, there always exist rotations $g_x \in O(3)$ that leave the axis going through the center of the sphere and $x$ invariant, that is $g_x x = x$ for those particular $g_x$. Then, at that point $x$, the equivariance condition \eqref{eq:equivariantA} rewrites
\begin{align*}
g_x \, A (g_x^{-1}x) = A(x) \quad \Rightarrow \quad g_x \, A(x) = A(x).
\end{align*}
This means that at that point $x$, $A(x) = f(x) x$, where $f(x)$ is any arbitrary function of $x$. Finally, if we consider any $g \in O(3)$ with $A$ having the above expression, we obtain
\begin{align*}
g \, f(g^{-1} x) \left( g^{-1} x \right) = f(g^{-1}x) x = f(x) x , \quad \text{ for all } g \in O(3).
\end{align*}
This means that $A(x) = f(r) x$ is a normal vector to the sphere, depending only of the radius of the sphere. Furthermore, we immediately notice that $A$ is a conservative field and therefore $\nabla \times A = B = 0$. We remark that this result was already obtained in \cite[Theorem 1.3]{ClappSzulkin}. Then physically, \eqref{eq:initialproblem} is equivalent to a problem without magnetic potential. In particular, the concentration on spheres of the solutions of \eqref{eq:initialproblem} is only driven by the scalar potential $V$ and we are exactly on the case studied by Ambrosetti, Malchiodi and Ni in \cite{AmbrosettiMalchiodiNi2003}.

\smallskip

A physically relevant case occurs in $\mathbb{R}^3$ in presence of magnetic and electric potentials having cylindrical symmetries. In that setting, we obtain a new surprising result for \eqref{eq:initialproblem}. We prove that the existence and the concentration of semiclassical bound states is influenced by the magnetic field when the concentration occurs on a circle.   We conjecture that this result should also occur in more general situations. More specifically, we consider the class of invariant scalar potentials and equivariant magnetic potentials under the action of the group
\begin{equation} \label{gruppo}
G:=\{g_\alpha \in O(3),\ \alpha \in [0, 2\pi[\,\},
\end{equation}
where
\begin{align*}
g_\alpha =
\begin{pmatrix}
\cos \alpha &- \sin \alpha & 0 \\
\sin \alpha &\cos \alpha & 0 \\
0 & 0 & \pm 1 \\
\end{pmatrix}.
\end{align*}
Namely, we assume that $A=(A_1, A_2, A_3) \in C^1(\mathbb{R}^3,\mathbb{R}^3)$ satisfies \eqref{eq:equivariantA} for every $g\in G$ given by \eqref{gruppo}. 
If we use the cylindrical coordinates $(x_1,x_2,x_3)= (\rho \cos \theta, \rho \sin \theta , x_3)$, the condition \eqref{eq:equivariantA} can be rewritten as
\begin{align*}
& A_1(\rho, \theta - \alpha, \pm x_3) =  \cos \alpha \, A_1 (\rho, \theta, x_3) + \sin \alpha \, A_2(\rho, \theta, x_3) \\
& A_2 (\rho, \theta - \alpha, \pm x_3) = -  \sin \alpha \, A_1(\rho, \theta , x_3) + \cos \alpha \, A_2 (\rho, \theta , x_3) \\
& A_3 (\rho, \theta - \alpha, \pm x_3) = \pm  A_3 (\rho, \theta, x_3).
\end{align*}
If we denote by
\begin{align*}
\textbf{e}_{\tau} = \left( - \sin \theta, \cos \theta, 0 \right), \qquad \textbf{e}_n = \left( \cos \theta, \sin \theta, 0 \right), \qquad \textbf{e}_3 = \left( 0, 0, 1 \right)
\end{align*}
an orthonormal basis of $\mathbb{R}^3$, we therefore infer that $A$ has the form
\begin{align*}
A(\rho, \theta, x_3) = \phi(\rho, |x_3|) \, \textbf{e}_{n} \, + \,  c(\rho, |x_3|) \, \textbf{e}_{\tau} \, + \, A_3(\rho, x_3) \, \textbf{e}_3, 
\end{align*}
for some functions $\phi, c \in C^1(\mathbb{R}^+ \times \mathbb{R}^+)$ and some $A_3 \in C^1(\mathbb{R}^+ \times \mathbb{R})$ which is odd in $x_3$.
The typical example $\phi \equiv 0 \equiv A_3$ and $c= b \rho / 2$, $b \in \mathbb{R} \backslash \{0\}$ corresponds to the constant magnetic field $B = b$ in the direction $x_3$ which is the simplest but also one of the more relevant case.

Next, we consider nonnegative cylindrically invariant potentials $V \in C(\mathbb{R}^3 \backslash \{0\})$, i.e. $V( g x)  = V(x) $ for every $g\in G$. This is equivalent to assume that $V$ depends only on $\rho$ and $|x_3|$. Moreover, we impose a growth condition at infinity when $p \in (2,4]$ :\\[-2mm]
\begin{itemize}
	\item[($V^\infty$)] there exists $\alpha \leq 2$ such that $\displaystyle \liminf_{|x| \to + \infty} V(x) |x|^\alpha > 0$.
\end{itemize}
When $p > 4$, we do not impose this restriction so that for instance one can deal with fast-decaying potentials or even compactly supported potentials. Nonetheless, as we will see later, we cannot consider $V \equiv 0$. 

\medbreak
At the origin, we do not require any specific assumptions on $V$. For instance, $V$ can behave singularly at the origin, or be locally bounded at the origin. However, if $V$ has a singularity, one can single out the Hardy potential as a threshold behaviour as in \cite{B-DC-VS}. If we assume in addition that \\[-2mm]
\begin{itemize}
	\item[($V^0$)] there exists $\alpha \geq 2$ such that $ \displaystyle \liminf_{|x|\to 0} V(x) |x|^\alpha > 0$,
\end{itemize}
then one can deduce a strong flatness of the solutions at the origin which depends on the order of the singularity.

\smallskip

We will look for solutions $u:\mathbb{R}^{3} \rightarrow\mathbb{C}$
of the problem
\begin{equation}
\left\{
\begin{array}
[c]{l}%
\left( i \hbar \nabla+A\right)  ^{2}u+Vu=\left\vert u\right\vert^{p-2}u, \\
u \in L^{2}(\mathbb{R}^{3},\mathbb{C}), \quad \left( i \hbar  \nabla + A \right) u \in L^{2}(\mathbb{R}^{3},\mathbb{C}^{3}).
\end{array}
\right.  \label{prop}%
\end{equation}
which satisfy the condition
\begin{equation*}
u(gx)= u(x)\text{ \ \ \ for all }g\in G,\text{ }x\in\mathbb{R}^{3},
\end{equation*}
and concentrate around a circle for small $\hbar > 0$. To this aim, we introduce the concentration function $\mathcal{M} \, : \, \mathbb{R}^+ \times \mathbb{R}^+ \rightarrow \mathbb{R}^+$ defined by 
\begin{align} \label{eq:concentration}
\mathcal{M}(\rho, |x_3|) = 2 \pi \rho \left[ c^2(\rho, |x_3|) + V(\rho,|x_3|) \right]^{\frac{2}{p-2}} \mathcal{E}(0,1),
\end{align}%
where  $\mathcal{E}(0,1)$ is a positive unrelevant constant  (see Section \ref{section:limit-problem} for more details).

Denoting by $\mathcal{H} \subset \mathbb{R}^3$ the $1$-dimensional vectorial subspace spanned by $\textbf{e}_3$, and by $\mathcal{H}^{\perp}$ its orthogonal complement, 
we assume the existence of a smooth bounded open $G$-invariant set $\Lambda \subset \mathbb{R}^3$ such that $\bar{\Lambda} \cap \mathcal{H} = \emptyset$, $\Lambda \cap \mathcal{H}^{\perp} \neq \emptyset$. By $G$-invariant, we mean that one has $g(\Lambda) = \Lambda$ for every $g \in G$.
Furthermore we assume that
\begin{align} \label{eq:condition_sur_Lambda}
\inf_{\Lambda \cap \mathcal{H}^{\perp}} \mathcal{M} <  \inf_{\partial \Lambda \cap \mathcal{H}^\perp} \mathcal{M} \quad \text{ and } \quad \inf_{\Lambda \cap \mathcal{H}^\perp} \mathcal{M} < 2 \inf_{\Lambda} \mathcal{M},
\end{align}
whereas
\begin{align} \label{eq:condition_sur_Lambda2}
\inf_{\bar{\Lambda}}V > 0.
\end{align}
Observe that the second assumption in \eqref{eq:condition_sur_Lambda} is in fact not restrictive if we take $\Lambda$  sufficiently small, since $\Lambda$ is smooth and $V$ is continuous in $\bar{\Lambda}$.

\smallskip
Using the facts that $A$ and $V$ have cylindrical symmetries, the equation in \eqref{prop} can be reduced to a problem in $\R^2$. Let $\rho_0 >0$ be a fixed radius and denote by ${A_0}:\mathbb{R}^2 \rightarrow \mathbb{R}^2$ the constant magnetic potential defined by  
\[
{A_0} = (\phi(\rho_0,0),0) \quad \text{ and } \quad  a_0 =  c^2(\rho_0,0) + V(\rho_0,0).
\]
We introduce the following two-dimensional problem 
\begin{align} \label{eq:limitproblem1}
( i \nabla + {A_0})^2 u + a_0 u = |u|^{p-2} u, \qquad  y=(y_1,y_2) \in \mathbb{R}^2
\end{align}
which can be regarded as a limiting problem for \eqref{prop}.   
Following the approach in \cite{B-DC-VS}, we will obtain the existence of cylindrically symmetric solutions of \eqref{eq:initialproblem} concentrating around circles in ${\Lambda \cap \mathcal{H}^\perp}$ for $\hbar>0$ small.

We stress that the two-dimensional limiting problem \eqref{eq:limitproblem1}, as well as the concentration function $\mathcal{M}$, takes into account the magnetic field  which will therefore influence the location of the concentration set of the semiclassical solutions of \eqref{eq:initialproblem}. This feature is new and, up to our knowledge, different from all the previous results in literature when dealing with an exterior magnetic field. 
Moreover, if we let $A_\tau (\rho) = A(\rho, \theta, 0) \cdot \textbf{e}_{\tau} = c(\rho, 0)$ be the tangential component of $A(\rho, \theta, 0)$ and  $A_n (\rho) = A(\rho, \theta, 0) \cdot \textbf{e}_n = \phi(\rho, 0)$ be its normal component, the solution of the two-dimensional limit problem \eqref{eq:limitproblem1} is given by $e^{i \left( A_n(\rho_0), 0 \right) \cdot y} w$, $y \in \mathbb{R}^2$, where $w$ is the ground state solution of 
\begin{align*}
- \Delta w + a_0 w = |w|^{p-2} w, \quad y \in \mathbb{R}^2.
\end{align*}
In this equation, $a_0 = c^2(\rho_0,0) + V(\rho_0,0) = A^2_\tau(\rho_0) + V(\rho_0,0)$. Therefore, our result below shows that the location of the concentration of the semiclassical bound states is influenced by the tangential component of $A$ and by the scalar potential $V$, while the phase factor of the semiclassical wave depends on the normal component of $A$. We conjecture that this is a general fact and that it is not just a consequence of the symmetry assumptions. 
 
\medbreak
 
In order to state our main result, we introduce some notations and tools adapted to the cylindrical symmetry of the problem. First, for $y,z \in \mathbb{R}^3$, we define the pseudometric
\begin{align*}
\mathrm{d}_{cyl}(y,z) = \left( (\rho_y - \rho_z)^2 +( y_3 - z_3)^2 \right)^{1/2},
\end{align*}
where $\rho_y = (y_1^2 + y_2^2)^{1/2}$ and $\rho_z = (z_1^2 + z_2^2)^{1/2}$. This function accounts for the distance between two circles. Then, for $r >0$ and $x \in \mathbb{R}^3$, we denote by $B_{cyl}(x, r)$ the ball (which is torus shaped)
\begin{align*}
B_{cyl}(x,r) = \left\{ y \in \mathbb{R}^3 \, | \, \mathrm{d}_{cyl}(x,y) < r \right\}.
\end{align*}

Our main theorem states, for $\hbar$ sufficiently small, the existence of solutions of $(\ref{prop})$ that concentrate around a circle $S^1_\hbar$ in the plane $x_{3}=0$, centered at the origin and of radius $\rho_\hbar$, where $\rho_{\hbar}$ converges to a minimizer of $\mathcal M$ in $\Lambda \cap \mathcal{H}^{\perp}$.

\begin{theorem} \label{theorem:main}
Let $p>2$. Let $V \in C(\mathbb{R}^3\backslash \{0\})$ and $A \in C^{1}(\mathbb{R}^3, \mathbb{R}^3)$ be such that $V (g x) = V(x)$ and $g \, A (g^{-1}x) = A(x)$, for every $g \in G$ defined in \eqref{gruppo}. Moreover, if $p\in (2,4)$, we suppose $V$ satisfies $(V^\infty)$. Assume that there exists a bounded smooth $G$-invariant set $\Lambda \subset \mathbb{R}^3$ such that \eqref{eq:condition_sur_Lambda} and \eqref{eq:condition_sur_Lambda2} are satisfied. Then there exists $\hbar_0 >0$ such that for every $0 <\hbar < \hbar_0$ 
\begin{itemize}
\item[(i)] the problem $(\ref{prop})$ has at least one solution $u_{\hbar} \in C^{1,\alpha}_{\text{loc}}(\mathbb{R}^3 \backslash \{0\})$ such that
$u_{\hbar} (g x) = u_{\hbar} (x)$ for every $g \in G$. 
\end{itemize}
Moreover, for every $0 <  \hbar <  \hbar_0$, $|u_{\hbar}|$ attains its maximum at some
$x_\hbar = (\rho_\hbar \cos \theta, \rho_\hbar \sin \theta, x_{3,\hbar}) \subset \Lambda$, $\theta \in [0, 2\pi [$, such that
\begin{itemize}
\item[(ii)]
$ \displaystyle \liminf_{\hbar \to 0} |u_{\hbar}(x_\hbar)| > 0.$
\item[(iii)]$ \displaystyle \lim_{\hbar \to 0} \mathcal{M}(x_\hbar) = \inf_{\Lambda \cap \mathcal{H}^\perp} \mathcal{M}$;
\item[(iv)] $\displaystyle \limsup_{\hbar \to 0} \frac{\mathrm{d}_{cyl}(x_\hbar, \mathcal{H}^\perp)}{\hbar} < + \infty$  , that is $x_{3,\hbar} \to 0$; \\
\item[(v)] $ \displaystyle \liminf_{\hbar \to 0} \mathrm{d}_{cyl}(x_\hbar, \partial \Lambda)> 0$.
\end{itemize}
Finally, 
for every   $0 <  \hbar < \hbar_0$  there exist  $C > 0$ and $\lambda > 0$ such that
the following asymptotic holds
\begin{itemize}
\item[(vi)] \hfil $\displaystyle 0 < |u_{\hbar}(x)| \leq C \exp \left( - \frac{\lambda}{\hbar} 
\frac{\mathrm{d}_{cyl}(x,x_\hbar)}{1 + \mathrm{d}_{cyl}(x,x_\hbar)} \right) (1 + |x|)^{-1} \quad \forall x \in \mathbb{R}^3\setminus\{0\}.$ \hfil
\end{itemize}
\end{theorem}

\begin{remark}
The last assertion {\rm (vi)} in Theorem \ref{theorem:main} combines a concentration estimate with a decay as $|x|\to\infty$. This decay at infinity is not enough to guarantee that our solutions are $L^2$ (since the ambiant space is $\R^3$).
However, this is only a rough estimate valid without further assumption on $V$ and it can be improved when assuming a slow decay of $V$ at infinity. Namely if we assume that ($V^\infty$) holds, then the solutions decay fast enough to be square integrable and thus they  are true bounded state solutions. 
We mention also that when ($V^0$) holds, we can estimate the flatness of the solution at the origin. We refer to Lemma \ref{lemma:barrier_function} for more details.
\end{remark}

\begin{example}
As a striking example, we observe that the presence of a constant magnetic field can produce a concentration phenomenon when coupled with a decaying electric potential. If we consider for instance the cubic nonlinearity, i.e. $p=4$, and the cylindrical Hardy potential $V(\rho)=1/\rho^2$,  $\rho^{2}=x_{1}^{2}+x_{2}^{2}$, there is no concentrated bound state (probably no bound state at all) of the equation without magnetic field. The presence of a constant magnetic field $B = b$ in the direction $x_3$ produces a solution that concentrates on the circle of radius $2^{1/2}/(3b^2)^{1/4}$.  
\end{example}

\medbreak
As a particular case of Theorem \ref{theorem:main}, we deduce the somewhat surprising result which states that when the scalar potential $V$ is constant, the existence and the location of our semiclassical states is only driven by the magnetic field.
\begin{corollary}
Assume $V \equiv \omega$, where $\omega$ is a positive constant. Then, under the assumptions of Theorem \ref{theorem:main}, the concentration of the solutions $u_{\hbar}$ holds at $$\inf_{\Lambda \cap \mathcal{H}^\perp} \mathcal{M} = \inf_{\Lambda \cap \mathcal{H}^\perp} 2 \pi \rho \left( c^2 + \omega \right)^{\frac{2}{p-2}} \mathcal{E}(0,1).$$
\end{corollary}

Finally, we also remark that Theorem \ref{theorem:main} do not require an upper bound on $p > 2$. Henceforth, we can treat critical and supercritical exponent problems by looking for cylindrically symmetric solutions of \eqref{eq:initialproblem}.

\medbreak

The paper is organized as follows. In Section \ref{section:variational-framework}, we give the variational framework and some related properties. Section \ref{section:limit-problem} is devoted to the study of the two dimensional limit problem. A penalization scheme is introduced in Section \ref{section:penalization-scheme} and  the existence of least-energy solutions is proved. The asymptotics of those solutions is studied in Section \ref{section:asymptotic-analysis}, while their concentration behaviour is established in Section \ref{section:solution-initial-problem}, showing that the solutions of the penalized problem solve the original one and therefore concluding the proof of Theorem \ref{theorem:main}. Finally, Section \ref{section:another-class-symmetric-solutions} is devoted to another class of symmetric solutions in the special case of a Lorentz type magnetic potential. These solutions are defined through an ansatz proposed by Esteban and Lions in \cite[Section 4.3]{EsLi}.

\section{The variational framework} \label{section:variational-framework}

In this section, we will fix our functional setting. In particular, we will define the Hilbert spaces adapted to the presence of a magnetic potential. We emphasize that in all Hilbert spaces we use, the scalar product will always be taken as the real scalar product, i.e. for every $z,w \in \mathbb{C}$, the scalar product will be defined by $\left(w | z \right) = \text{Re} (w \bar{z})$.

\subsection{The magnetic spaces}

Let $N \geq 2$. For $A \in L^2_{\text{loc}}(\mathbb{R}^N, \mathbb{R}^N)$, we define the space $\mathcal{D}^{1,2}_{A,\varepsilon}(\mathbb{R}^N,\mathbb{C})$ as the closure of $C^{\infty}_{0}(\mathbb{R}^N,\mathbb{C})$ with respect to the norm defined through
\begin{align*}
\|u\|_{\mathcal{D}^{1,2}_{A,\varepsilon}}^2:=\int_{\mathbb{R}^N} |(i \varepsilon \nabla + A) u|^2.
\end{align*}
Similarly, $\mathcal{D}^{1,2}(\mathbb{R}^N,\mathbb{C})$ (resp. $\mathcal{D}^{1,2}(\mathbb{R}^N,\mathbb{R})$) is the closure of $C^{\infty}_{0}(\mathbb{R}^N,\mathbb{C})$ (resp. $C^{\infty}_{0}(\mathbb{R}^N,\mathbb{R})$) with respect to the norm defined through
\begin{align*}
\|u\|_{\mathcal{D}^{1,2}}^2:=\int_{\mathbb{R}^N} |\nabla u |^2.
\end{align*}
Remember that the Sobolev inequality implies that $\mathcal{D}^{1,2}(\mathbb{R}^N,\mathbb{C})$ (resp. $\mathcal{D}^{1,2}(\mathbb{R}^N,\mathbb{R})$) is embedded in $L^{2^\star}(\mathbb{R}^N, \mathbb{C})$ (resp. $L^{2^\star}(\mathbb{R}^N, \mathbb{R})$). 
We also consider the space
\begin{align*}
H^1_{A,\varepsilon}(\mathbb{R}^N, \mathbb{C}) = \left\{ u \in L^2(\mathbb{R}^N, \mathbb{C}) \, | \, (i \varepsilon \nabla + A)u \in L^{2}(\mathbb{R}^N, \mathbb{C}^N) \right\},
\end{align*}
endowed with the norm
\begin{align*}
\| u \|_{H^1_{A,\varepsilon}}^2 = \int_{\mathbb{R}^N}  |(i \varepsilon \nabla + A) u|^2 + |u|^2 .
\end{align*}
We remark that, in general, this space is not embedded in $H^1(\mathbb{R}^N,\mathbb{C})$ (and inversely). However if $u \in H^1_{A,\varepsilon}(\mathbb{R}^N, \mathbb{C})$, then $|u| \in H^1(\mathbb{R}^N, \mathbb{R})$. This is the diamagnetic inequality that we recall here.
\begin{lemma}[Diamagnetic inequality]
Let $A : \mathbb{R}^N \rightarrow \mathbb{R}^N$ be in $L^2_{\text{loc}}(\mathbb{R}^N, \mathbb{R}^N)$ and let $u \in \mathcal{D}^{1,2}_{A,\varepsilon}(\mathbb{R}^N, \mathbb{C})$. Then, $|u| \in \mathcal{D}^{1,2}(\mathbb{R}^N, \mathbb{R})$ and the diamagnetic inequality
\begin{align} \label{eq:diamagnetic}
\varepsilon \left|\nabla |u|(x) \right| \leq \left| (i \varepsilon \nabla + A) u(x) \right|
\end{align}
holds for almost every $x \in \mathbb{R}^N$ and for every $\varepsilon > 0$.
\end{lemma}
\begin{proof}
We compute
\begin{align*}
\varepsilon \nabla |u| = \text{Im} \left( i \varepsilon \nabla u \frac{\bar{u}}{|u|} \right) = \text{Im}\left((i \varepsilon \nabla + A)u \, \frac{\bar{u}}{|u|} \right) \quad \text{ a.e. }
\end{align*}
because $A$ is real-valued. We conclude using the fact that $|\text{Im}(z)| \leq |z|$ for any complex number $z$.
\end{proof}
Using \eqref{eq:diamagnetic}, we can verify that, for every $u \in \mathcal{D}^{1,2}_{A,\varepsilon}(\mathbb{R}^N,\mathbb{C})$,
\begin{align} \label{eq:diamagnetic_epsilon}
\varepsilon^2 \int_{\mathbb{R}^N} | \nabla |u||^2 \, \mathrm{d}x \leq \int_{\mathbb{R}^N} | (i \varepsilon \nabla + A) u|^2 \, \mathrm{d}x,
\end{align}
for any $\varepsilon > 0$. 

\medbreak

In some particular cases, the spaces $H^1_{A,\varepsilon}(\mathbb{R}^N,\mathbb{C})$ and $H^1(\mathbb{R}^N,\mathbb{C})$ are equivalent as for instance if $A$ is bounded. The following lemma is proved for example in \cite[Lemma 3.1]{CingolaniSecchi2002}.
\begin{lemma} \label{lemma:equivalence_spaces}
Let $A : \mathbb{R}^N \rightarrow \mathbb{R}^N$ be such that $|A| \leq C$ for all $x \in \mathbb{R}^N$, $C \geq 0$. Then, the spaces $H^1_{A,\varepsilon}(\mathbb{R}^N, \mathbb{C})$ and $H^1(\mathbb{R}^N,\mathbb{C})$ are equivalent.
\end{lemma}

Finally, we introduce the following Hilbert space
\begin{align*}
H^1_{A,V,\varepsilon}(\mathbb{R}^N,\mathbb{C}) = \left\{ u \in \mathcal{D}^{1,2}_{A,\varepsilon}(\mathbb{R}^N,\mathbb{C}) \, | \, \int_{\mathbb{R}^N} V(x) |u|^2 < \infty \right\},
\end{align*}
endowed with the norm
\begin{align*}
\| u \|^2_{H^1_{A,V,\varepsilon}} = \int_{\mathbb{R}^N} | (i \varepsilon \nabla + A) u|^2 + V(x) |u|^2.
\end{align*}
In what follows, for simplicity, we write  $\|u\|_\varepsilon$ instead of $\| u \|_{H^1_{A,V,\varepsilon}}$.

\subsection{Hardy and Kato inequalities}

In dimensions $N \ge 3$, the Hardy inequality for functions $u \in \mathcal{D}^{1,2}(\mathbb{R}^N,\mathbb{C})$ writes 
\begin{align*} 
\left( \frac{N-2}{2} \right)^2 \int_{\mathbb{R}^N} \frac{|u(x)|^2}{|x|^2} \, \mathrm{d}x \leq \int_{\mathbb{R}^N} |\nabla u|^2,
\end{align*}
and for functions $u \in \mathcal{D}^{1,2}_{A,\varepsilon}(\mathbb{R}^N,\mathbb{C})$
\begin{align} \label{eq:Hardy_inequality_magnetic}
\varepsilon^2 \left( \frac{N-2}{2} \right)^2 \int_{\mathbb{R}^N} \frac{|u(x)|^2}{|x|^2} \, \mathrm{d}x \leq \varepsilon^2 \int_{\mathbb{R}^N} |\nabla |u||^2 \leq \int_{\mathbb{R}^N} | ( i \varepsilon \nabla + A)u|^2,
\end{align}
for any $\varepsilon > 0$. Furthermore, we recall the following Kato's inequalities. First, for functions $u \in L^1_{\text{loc}}(\mathbb{R}^N,\mathbb{C})$ with $\nabla u \in L^1_{\text{loc}}(\mathbb{R}^N, \mathbb{C}^N)$, we define
\begin{align*}
\text{sign}(u)(x) = \left\{ \begin{aligned} & \frac{\bar{u}(x)}{|u(x)|} \quad & u(x) \neq 0 \\
                                           & 0 \quad & u(x) = 0.
                                           \end{aligned} \right.
\end{align*}
We have
\begin{align} \label{eq:kato_inequality}
\Delta |u| \geq  \text{Re} \left( \text{sign}(u) \Delta u \right).
\end{align}
We also have a similar inequality in presence of a magnetic potential $A \in L^2_{\text{loc}}(\mathbb{R}^N, \mathbb{R}^N)$,
\begin{align} \label{eq:magnetic_kato_inequality}
\varepsilon^2 \Delta |u| \geq - \text{Re} \left(\text{sign}(u) (i \varepsilon \nabla + A)^2 u \right).
\end{align}

Throughout the text, we will use an auxiliary Hardy type potential. This potential was first introduced in \cite{MorozVanschaftingen2009-2,MorozVanschaftingen2009-1} to extend the penalization method of del Pino and Felmer to compactly supported potentials $V$. For $N \ge 3$, we define the function $H : \mathbb{R}^N \rightarrow \mathbb{R}$ by
\begin{align*}
H(x) = \frac{\kappa}{|x|^2 \left( (\log|x|)^2 + 1 \right)^{\frac{1+\beta}{2}}},
\end{align*}
for $\beta > 0$ and $0 < \kappa < \left(\frac{N-2}{2}\right)^2$. Notice that, for all $x \in \mathbb{R}^N$, we have
\begin{align} \label{eq:inequality_on_H}
H(x) \leq \frac{\kappa}{|x|^2}, \quad \text{ or } \quad H(x) \leq \frac{\kappa}{|x|^2 \left| \log |x| \right|^{1 + \beta}}.
\end{align}
The interest of this auxiliary potential comes mainly from the following comparison principle for $-\Delta - H$ which was proved in \cite{B-DC-VS}.
\begin{lemma} \label{lemma:comparison}
Let $N \geq 3$ and $\Omega \subset \mathbb{R}^N \setminus \{0\}$ be a smooth domain. Let $v,w \in H^1_{\text{loc}}(\Omega, \mathbb{R})$ be such that $\nabla (w-v)_{-} \in L^{2}(\Omega)$, $(w-v)_{-} / |x| \in L^{2}(\Omega)$ and 
\begin{align*} 
- \Delta w - H(x) w \geq -\Delta v - H(x) v, \quad \forall x \in \Omega.
\end{align*}
Moreover, if $\partial \Omega \neq \emptyset$, assume that $w \geq v $ on $\partial \Omega$. Then, $w \geq v$ in $\Omega$.
\end{lemma}
We also point out that $H$ can be compared to the Hardy potential $C/|x|^2$ which is a critical potential, both at zero and at infinity. Indeed, if $V$ behaves like $1/|x|^\alpha$ at infinity, for $\alpha \geq 2$, then we have an equivalence between $H^1_{A,V,\varepsilon}(\mathbb{R}^N,\mathbb{C})$ and $\mathcal{D}^{1,2}_{A,\varepsilon}(\mathbb{R}^N,\mathbb{C})$. It means that the condition $\int_{\mathbb{R}^N} V(x) |u|^2 < \infty$ is unnecessary in that case. The same holds true if $V$ is singular at $0$ and behaves as $1/|x|^\alpha$ for $\alpha \leq 2$.

\subsection{Notations adapted to the cylindrical symmetry}

From now on, we deal with dimension $N=3$. As we will work with functions having cylindrical symmetry, that is functions such that $u \circ g = u$ for $g\in G$, where $G$ is defined in $(\ref{gruppo})$, the significant variables are $\rho \in \mathbb{R}^+$ and $x_3 \in \mathbb{R}$, where $\rho = (x_1^2 + x_2^2)^{1/2}$, $x=(x_1,x_2,x_3)$. The angular variable $\theta \in [0,2\pi)$ plays no role. However, even if those functions only depend on $\rho$ and $x_3$, there are still functions defined in $\mathbb{R}^3$ and with some abuse of notations we will write either $u(\rho, x_3)$ or $u(x_1,x_2,x_3)$ depending on the situation. 

We also recall the distance adapted to cylindrical symmetry already mentioned in the introduction: for $y, z \in \mathbb{R}^3$,
\begin{align*} 
\mathrm{d}_{cyl}(y,z) = \left( (\rho_y - \rho_z)^2 +( y_3 - z_3)^2 \right)^{1/2},
\end{align*}
for $\rho_y = (y_1^2 + y_2^2)^{1/2}$ and $\rho_z = (z_1^2 + z_2^2)^{1/2}$, as well as the cylindrical ball
\begin{align*} 
B_{cyl}(x,r) = \left\{ y \in \mathbb{R}^3 \, | \, \mathrm{d}_{cyl}(x,y) < r \right\},
\end{align*}
for $r > 0$ and $x \in \mathbb{R}^3$.

The following lemma gives us some compact embedding of the magnetic Sobolev spaces with cylindrical symmetry.
\begin{lemma} \label{lemma:cylindrical-embeddings-compact}
Assume that $\Omega \subset \mathbb{R}^3$ is an open bounded set such that
\begin{align*}
g(\Omega) = \Omega \text{ for every }g\in G \quad \text{ and } \quad 0 < \rho_0 < \rho < \rho_1 \quad \text{for every } (\rho \cos \theta, \rho \sin \theta, x_3)\in\Omega . 
\end{align*}
Then, the space $\left\{ u \in H^1_{A,\varepsilon}(\Omega,\mathbb{C}) \, | \, u \circ g = u \text{ for every }g\in G \right\}$ is compactly embedded in $L^q(\Omega)$, for $2 \leq q < + \infty$.
\end{lemma}

\begin{proof}
First, since $\Omega$ is bounded and $A \in C^1(\mathbb{R}^3,\mathbb{R}^3)$, we have seen in Lemma \ref{lemma:equivalence_spaces} that this space is equivalent to $\left\{ u \in H^1(\Omega, \mathbb{C}) \, | \, u \circ g = u \text{ for every }g\in G \right\}$. Since $u$ depends only on $\rho$ and $x_3$, we can write the square of the $H^1$-norm of $u$ as
\begin{align*}
\int_{\Omega}  \left( |\nabla u|^2 + |u|^2 \right) \, \mathrm{d}x_1\mathrm{d}x_2\mathrm{d}x_3 = 2\pi \int_{\Omega_0} \left( \left| \partial_\rho u\right|^2 + \left|\partial_{x_3} u \right|^2 + |u|^2 \right) \, \rho \, \mathrm{d}\rho \mathrm{d}x_3,
\end{align*}
where $\Omega_0$ is the parametrization of $\Omega$ in the $\rho,x_3$ variables. Now, take a bounded sequence $(u_n)_n\subset H^1_{A,\varepsilon}(\Omega,\mathbb{C})$. Considering each $u_n$ as a function of the two variables $\rho,x_3$ in $\mathbb{R}^2$, we infer that the sequence is bounded as a sequence $(u_n)_n\subset H^1(\Omega_0)$. We can then use the compact embedding in dimension $2$ to conclude.
\end{proof}

\section{The limit problem} \label{section:limit-problem}

Because of the symmetry, our solutions will concentrate on circles and the limit problem will hold in $\mathbb{R}^2$. 
The aim of this section is to describe such a limit problem. Consider a constant potential ${A_0}:\mathbb{R}^2 \rightarrow \mathbb{R}^2$ and a positive constant $a_0 > 0$. The equation
\begin{align} \label{eq:limitproblem}
( i \nabla + {A_0})^2 u + a_0 u = |u|^{p-2} u, \quad   y=(y_1,y_2) \in \mathbb{R}^2
\end{align}
will be referred to as the limit equation associated to the problem \eqref{eq:initialproblem}. For solutions concentrating around a circle of radius $\rho_0>0$, we will have
\begin{align*}
{A_0} = (\phi(\rho_0,0),0) \quad \text{ and } \quad  a_0 =  c(\rho_0,0)^2 + V(\rho_0,0).
\end{align*}  
By lemma \ref{lemma:equivalence_spaces}, the weak solutions of \eqref{eq:limitproblem} are critical points of the functional $\mathcal{J}_{a_0}^{A_0}: H^1(\mathbb{R}^2, \mathbb{C}) \rightarrow \mathbb{R}$ defined by
\begin{align} \label{eq:limit_functional}
\mathcal{J}_{a_0}^{A_0}(u) = \frac{1}{2} \int_{\mathbb{R}^2} \left[ |(i \nabla + A_0) u|^2 + a_0 |u|^2 \right] \, \mathrm{d}y - \frac{1}{p} \int_{\mathbb{R}^2} |u|^p \, \mathrm{d}y.
\end{align}
Any nontrivial critical point $u \in H^1(\mathbb{R}^2, \mathbb{C})$ of $J_{a_0}^{A_0}$ belongs to the Nehari manifold
\begin{align*} 
\mathcal{N}_{a_0}^{A_0} = \left\{ u \in H^1(\mathbb{R}^2,\mathbb{C}) \, |  \, u \not\equiv 0 \text{ and } \langle (\mathcal{J}_{a_0}^{A_0})^\prime (u), u \rangle = 0 \right\}.
\end{align*}
A solution $u \in H^1(\mathbb{R}^2, \mathbb{C})$ is called a least energy solution, or ground state, of \eqref{eq:limitproblem} if 
\begin{align*}
\mathcal{J}_{a_0}^{A_0}(u) = \inf_{v \in \mathcal{N}_{a_0}^{A_0}} \mathcal{J}_{a_0}^{A_0}(v).
\end{align*}

The following lemma states that any least energy solution of the limit problem \eqref{eq:limitproblem} is real up to a change of gauge and a complex phase.
\begin{lemma} \label{lemma:least-energy-solution-limit-problem}
Suppose $v$ is a least energy solution of equation \eqref{eq:limitproblem}. Then
\begin{align*}
v(y) = w(y - y_0) e^{i \alpha} e^{i A_0 \cdot y},
\end{align*}
for some $\alpha \in \mathbb{R}$, $y_0 \in \mathbb{R}^2$ and where $w$ is the unique radially symmetric real positive solution of the scalar equation
\begin{align} \label{eq:reallimitequation}
- \Delta w + a_0 w = |w|^{p-2}w \quad \text{ in } \mathbb{R}^2.
\end{align}
\end{lemma}

\begin{proof}
First, we consider the functional $\mathcal{J}_{a_0}^0 : H^1(\mathbb{R}^2, \mathbb{C}) \rightarrow \mathbb{R}$ associated to equation \eqref{eq:reallimitequation}
\begin{align*} 
\mathcal{J}_{a_0}^0 (u) = \frac{1}{2} \int_{\mathbb{R}^2} |\nabla u|^2 + a_0 |u|^2 \, \mathrm{d}y - \frac{1}{p} \int_{\mathbb{R}^2} |u|^p \, \mathrm{d}y.
\end{align*}
Again, any nontrivial critical point $u \in H^1(\mathbb{R}^2,\mathbb{C})$ of $\mathcal{J}_{a_0}^0$ belongs to the Nehari manifold $\mathcal{N}_{a_0}^0$. By performing the change of gauge
\begin{align} \label{eq:gauge}
v(y) = e^{i {A_0} \cdot y} u(y)
\end{align}
on functions $v \in \mathcal{N}_{a_0}^{{A_0}}$ and $u \in \mathcal{N}_{a}^0$, we observe that there is an isomorphism between the two Nehari manifolds. Indeed, any least energy solution $v$ of $\mathcal{J}_{a_0}^{A_0}$ provides a least energy solution $u$ of $\mathcal{J}_{a_0}^0$ by \eqref{eq:gauge} and vice-versa.

Since it is well-known, see for example \cite[Lemma 7]{Kurata}, that the set of complex valued least energy solutions $u$ of $\mathcal{J}_{a_0}^0$ can be written as
\begin{align*}
\{u(y) = e^{i \alpha} w(y - y_0),\ \alpha \in \mathbb{R},\ y_0 \in \mathbb{R}^2\},
\end{align*}
the proof is completed.
\end{proof}

We may now define the ground energy function $\mathcal{E} \, : \, \mathbb{R}^2 \times \mathbb{R}^+ \backslash \{0\} \rightarrow \mathbb{R}^+ $ by
\begin{align*}
\mathcal{E}(A_0,a_0) = \inf_{v \in \mathcal{N}_{a_0}^{A_0}} J_{a_0}^{A_0}(v).
\end{align*}
The following lemma gives some properties of this ground energy function. We refer to \cite{B-DC-VS} or \cite{Rabinowitz} for more details.

\begin{lemma}
For every $(A_0,a_0) \in \mathbb{R}^2 \times \mathbb{R}^+ \backslash \{0\}$, $\mathcal{E}({A_0},a_0)$ is a critical value of $\mathcal{J}^{A_0}_{a_0}$ and we have the following variational characterization 
\begin{align*}
\mathcal{E}(0,a_0) = \mathcal{E}(A_0,a_0) = \inf_{v \in H^1(\mathbb{R}^2,\mathbb{C}) \backslash \{ 0 \}}  \max_{t \geq 0} \mathcal{J}^{A_0}_{a_0}(v).
\end{align*}
Moreover, 
\begin{itemize}
\item[(i)] for every ${A_0} \in \mathbb{R}^2$, $a_0 \in \mathbb{R}^+ \backslash \{0\} \mapsto \mathcal{E}({A_0},a_0)$ is continuous;
\item[(ii)] for every ${A_0} \in \mathbb{R}^2$, $a_0  \in \mathbb{R}^+ \backslash \{0\}  \mapsto \mathcal{E}({A_0},a_0)$ is strictly increasing.
\end{itemize}
In fact, for our nonlinearity 
\begin{align} \label{eq:croissance_ground_energy}
\mathcal{E}(0,a_0) = \mathcal{E}({A_0},a_0) = \mathcal{E}({A_0} a_0^{- \frac{1}{p-2}},1) a_0^{\frac{2}{p-2}} = \mathcal{E}(0,1) a_0^{\frac{2}{p-2}}.
\end{align}
\end{lemma}
Finally, the concentration function $\mathcal{M} \, : \, \mathbb{R}^+ \times \mathbb{R}^+ \rightarrow \mathbb{R}^+$, already introduced in \eqref{eq:concentration}, is defined more precisely by
\begin{align*} 
\mathcal{M}(\rho, |x_3|) = 2 \pi \rho \, \mathcal{E}(0, c^2(\rho,|x_3|) + V(\rho,|x_3|)) = 2 \pi \rho \left[ c^2(\rho, |x_3|) + V(\rho,|x_3|) \right]^{\frac{2}{p-2}} \mathcal{E}(0,1) .
\end{align*}

We will look for solutions concentrating around local minima of $\mathcal{M}$.

\section{The penalization scheme}  \label{section:penalization-scheme}

The functional associated to equation \eqref{eq:initialproblem} is given by
\begin{align*}
\int_{\mathbb{R}^3} \left( | (i \varepsilon \nabla + A) u|^2 + V |u|^2 \right) \, \mathrm{d}x - \frac{1}{p} \int_{\mathbb{R}^3} |u|^p \, \mathrm{d}x.
\end{align*}
It is natural to consider this functional in the Sobolev space $H^1_{A,V,\varepsilon}(\mathbb{R}^3, \mathbb{C})$. However, the mere assumptions on $V$, and more particularly the fact that $V$ can decay to zero at infinity, do not ensure that $H^1_{A,V,\varepsilon}(\mathbb{R}^3,\mathbb{C})$ is embedded in the $L^{p}(\mathbb{R}^3,\mathbb{C})$. Then, the last term of the functional is not necessarily finite.
Moreover, even if we assume that $V$ is bounded away from zero, the functional would have a mountain-pass geometry in $H^1_{A,V,\varepsilon}(\mathbb{R}^3,\mathbb{C})$, but the Palais-Smale condition could fail without further specific assumptions on $V$. 

For those reasons, following del Pino and Felmer \cite{DP-F}, we truncate the nonlinear term through a penalization outside the set where the concentration is expected. Basically, the penalization approach consists in modifying the nonlinearity outside the bounded set $\Lambda$, where $\Lambda$ verifies \eqref{eq:condition_sur_Lambda} and \eqref{eq:condition_sur_Lambda2}, in the following way
\begin{align*}
\tilde{f}(x,s) = \min \{ \mu V(x) s, f(s) \},
\end{align*}
where $0 < \mu < 1$. The penalized functional, given by
\begin{align*}
\int_{\mathbb{R}^3} \left( | (i \varepsilon \nabla + A) u|^2 + V |u|^2 \right) \, \mathrm{d}x - \int_{\mathbb{R}^3} \tilde F(|u|) \, \mathrm{d}x,
\end{align*}
where $\tilde F(\tau) = \int_0^\tau \tilde f(s) \, ds$, has the mountain-pass geometry and we recover the Palais-Smale condition thanks to the penalization, so that we can easily deduce the existence of a mountain pass critical point $u$. Then, if we succeed to show that $f(u) \leq \nu V(x)u$ outside the set $\Lambda$, we recover a solution of the initial problem.

We will argue slightly differently for two reasons. The first one is that this approach works fine when $V$ stays bounded away from zero, or at least when $V$ does not converge to fast to zero at infinity. We do not want to restrict our assumptions on the potentials $V$ to this class. To solve this issue, we will add the term $\varepsilon^2 H(x)$ to $V$ in the modified nonlinearity. This penalization approach was first introduced in \cite{MorozVanschaftingen2009-2,MorozVanschaftingen2009-1} as an extension of \cite{BonheureVanschaftingen} and subsequently used in \cite{B-DC-VS}. The second reason, as already said before, is that our functions are complex-valued. We will then perform the penalization on the modulus of the unknown.

\subsection{The penalized functional}

We fix $\mu \in (0,1)$. We define the penalized nonlinearity $g_{\varepsilon} \, : \, \mathbb{R}^3 \times \mathbb{R}^+ \rightarrow \mathbb{R}$ by
\begin{align} \label{eq:penalized_nonlinearity}
g_{\varepsilon}(x,s) = \chi_{\Lambda}(x) f(s) + (1 - \chi_{\Lambda}(x)) \min\left\{ (\varepsilon^2 H(x) + \mu V(x) ) , f(s) \right\}
\end{align}
for $f(s) = s^{\frac{p-2}{2}}$. Let $G_{\varepsilon}(x,s) = \frac{1}{2} \int_{0}^{s} g_{\varepsilon}(x,\sigma) \, \mathrm{d}\sigma$. There exists $2 < \theta \leq p$ such that
\begin{align} \label{eq:property_g1}
 0 < \theta G_{\varepsilon}(x,s) \leq g_{\varepsilon}(x,s) s & \quad \quad \forall x \in \Lambda, \, \forall s> 0 ,\\ \label{eq:property_g2}
 0 < 2 G_{\varepsilon}(x,s) \leq g_{\varepsilon}(x,s) s \leq (\varepsilon^2 H(x) + \mu V(x) ) s & \quad \quad \forall x \notin \Lambda , \, \forall s > 0.
\end{align} 
Moreover, we have that 
\begin{align} \label{eq:property_g3}
g_\varepsilon(x,s^2) \quad  \text{ is nondecreasing } \quad \forall x \in \mathbb{R}^3,
\end{align}
which is a useful property, see for example \cite{Rabinowitz}.

In the following we look for  cylindrically symmetric solutions of the penalized equation
\begin{align} \label{eq:penalized_problem}
(i \varepsilon \nabla + A)^2 u + V(x) u = g_{\varepsilon}(x,|u|^2)u, \quad x \in \mathbb{R}^3.
\end{align}
Let us define the penalized functional $\mathcal{J}_{\varepsilon} : H^1_{A,V,\varepsilon}(\mathbb{R}^3,\mathbb{C}) \rightarrow \mathbb{R}$
\begin{align*}
\mathcal{J}_{\varepsilon}(u) = \frac{1}{2} \int_{\mathbb{R}^3} | (i \varepsilon \nabla + A) u|^2 + V(x) |u|^2 - \int_{\mathbb{R}^3} G_{\varepsilon}(x, |u|^2),
\end{align*}
and the space
\begin{align*}
\mathcal{X}_\varepsilon = \left\{ u \in H^1_{A,V,\varepsilon}(\mathbb{R}^3,\mathbb{C}) \, | u \circ g = u, \ \forall g \in G \right\}.
\end{align*}
By the principle of symmetric criticality \cite{Palais}, the critical points of $\mathcal{J}_\varepsilon$ in $\mathcal{X}_\varepsilon$ are weak solutions of the penalized problem \eqref{eq:penalized_problem}, having cylindrical symmetry. Thanks to the properties \eqref{eq:property_g1} and \eqref{eq:property_g2}, the functional has a mountain pass geometry. Indeed, it clearly displays a local minimum at $u=0$, while the infimum is $- \infty$. Standard arguments imply then the existence of a Palais-Smale sequence $(u_n)_n \subset \mathcal{X}_\varepsilon$ for $\mathcal{J}_{\varepsilon}$, that is
\begin{align*}
\mathcal{J}_{\varepsilon}(u_n) \leq C \quad \text{ and } \quad \mathcal{J}_{\varepsilon}^{\prime}(u_n) \rightarrow 0 \text{ as } n \to \infty.
\end{align*}
To secure the existence of a weak solution of \eqref{eq:penalized_problem} for every $\varepsilon>0$, it only remains to prove that $\mathcal{J}_{\varepsilon}$ satisfies the Palais-Smale condition, i.e. each Palais-Smale sequence possesses a convergent subsequence. This is our next aim.

\subsection{The Palais-Smale condition}

\begin{lemma} \label{lemma:existence_solutions_penalized_problem}
For every $\varepsilon > 0$, every Palais-Smale sequence for $\mathcal{J}_{\varepsilon}$ in $\mathcal{X}_\varepsilon$ contains a convergent subsequence.
\end{lemma}

\begin{proof} We proceed in several steps. 

\medbreak

\textbf{Step 1.} As usual, the first step of the proof consists in proving that the Palais-Smale sequence $(u_n)_n$ is bounded. By using successively the properties of the Palais-Smale sequence $(u_n)_n$, \eqref{eq:property_g1}, \eqref{eq:property_g2} and finally the magnetic Hardy inequality \eqref{eq:Hardy_inequality_magnetic}, we infer that 
\begin{align*}
\frac{1}{2} \|u_n\|_\varepsilon^2 & = \mathcal{J}_{\varepsilon}(u_n) + \int_{\mathbb{R}^3} G_{\varepsilon}(x,|u_n|^2) \\
& \leq C + \int_{\Lambda} G_{\varepsilon}(x,|u_n|^2) + \int_{\Lambda^{c}} G_{\varepsilon}(x,|u_n|^2) \\
& \leq C + \frac{1}{\theta} \int_{\mathbb{R}^3} g_{\varepsilon}(x, |u_n|^2) |u_n|^2 + \left( \frac{1}{2} - \frac{1}{\theta} \right) \int_{\Lambda^c} g_{\varepsilon}(x,|u_n|^2) |u_n|^2 \\
& \leq C + \frac{1}{\theta} \|u_n\|_\varepsilon^2 - \frac{1}{\theta} \langle \mathcal{J}_{\varepsilon}'(u_n), u_n \rangle + \left( \frac{1}{2} - \frac{1}{\theta} \right) \int_{\Lambda^c} \left( \varepsilon^2 H(x) + \mu V(x) \right) |u_n|^2 \\
& \leq C + \frac{1}{\theta} \|u_n\|_\varepsilon^2 + o(1) \|u_n\|_\varepsilon + \left(\frac{1}{2} - \frac{1}{\theta} \right) \mu \int_{\mathbb{R}^3} V(x) |u_n|^2 + \left( \frac{1}{2} - \frac{1}{\theta} \right) 4 \kappa \int_{\mathbb{R}^3} | (i \varepsilon \nabla + A) u_n|^2 \\
& \leq C + \frac{1}{\theta} \| u_n\|_\varepsilon^2 + o(1) \|u_n\|_\varepsilon + \left( \frac{1}{2} - \frac{1}{\theta} \right) \max\{\mu, 4 \kappa \} \|u_n\|_\varepsilon^2.
\end{align*}
Since $\theta > 2$ and $\max\{ \mu, 4 \kappa \} < 1$, the inequality 
\begin{align*}
\left(\frac{1}{2} - \frac{1}{\theta} \right) \left( 1 - \max\{ \mu, 4 \kappa \} \right) \|u_n\|_\varepsilon^2 \leq C + o(1) \|u_n\|_\varepsilon
\end{align*}
leads to the conclusion. 

\medbreak

From Step 1, we deduce the existence of a function $u \in \mathcal{X}_\varepsilon$ such that, up to a subsequence still denoted in the same way, $u_n$ weakly converges to $u$.

\medbreak

\textbf{Step 2.} In this step, we prove two useful claims aiming to deduce the strong convergence. We define the closed set $A_\lambda = \overline{B(0,e^\lambda)} \backslash B(0,e^{-\lambda})$, where $\lambda \geq 0$. 

\smallbreak

\textit{Claim 1 - for every $\delta > 0$, there exists $\lambda_\delta \geq 0$ such that
\begin{align} \label{eq:claim1}
\limsup_{n \to \infty} \varepsilon^2 \int_{\mathbb{R}^3 \backslash A_{\lambda_\delta}} H(x) |u_n|^2 < \delta.
\end{align}}
The inequality \eqref{eq:inequality_on_H} together with Hardy inequality \eqref{eq:Hardy_inequality_magnetic} yields
\begin{align*}
\int_{\mathbb{R}^3 \backslash A_{\lambda}} H(x) |u_n|^2 \leq \frac{4 \kappa}{\lambda^{1+\beta}} \int_{\mathbb{R}^3} | (i \varepsilon \nabla + A) u_n|^2.
\end{align*}
Since $(u_n)_n$ is bounded,  we now infer that for every $\delta >0$, there exists $\lambda_\delta \geq 0$ such that \eqref{eq:claim1} holds.

\medbreak

\textit{Claim 2 - for every $\delta > 0$, there exists $\lambda_\delta \geq 0$ (eventually bigger than the previous one) such that
\begin{align} \label{eq:claim2}
\limsup_{n \to \infty} \int_{\mathbb{R}^3 \backslash A_{\lambda_\delta}} V(x) |u_n|^2 < \delta. 
\end{align}}
We first define $\xi \in C^{\infty}(\mathbb{R})$ such that $0 \leq \xi \leq 1$ and 
\begin{align*}
\xi(s) = \left\{ 
\begin{array}{ll} 
0 \quad & \text{ if }|s| \leq \frac{1}{2} \\
1 \quad & \text{ if } |s| \geq 1
\end{array} \right.
\end{align*}
to build the cut-off function $\eta_{\lambda} \in C^{\infty}(\mathbb{R}^3,\mathbb{R})$ as
\begin{align*}
\eta_\lambda (x) = \xi\left( \frac{\log|x|}{\lambda} \right).
\end{align*}
Since $(u_n)_n$ is a bounded Palais-Smale sequence and $\eta_{\lambda} \leq 1$, we deduce that $\langle \mathcal{J}^{\prime}_{\varepsilon}(u_n), \eta_\lambda u_n \rangle = o(1)$. We then infer that
\begin{multline} \label{eq:eq1}
\int_{\mathbb{R}^3} \left( | (i \varepsilon \nabla + A)u_n|^2 + V(x) |u_n|^2 \right) \eta_\lambda \\ = \int_{\mathbb{R}^3} g_{\varepsilon}(x, |u_n|^2) |u_n|^2 \eta_\lambda + \text{Re} \int_{\mathbb{R}^3} i \varepsilon (i \varepsilon \nabla + A) u_n \cdot  \nabla \eta_\lambda \overline{u_n} + o(1).
\end{multline}
Since $\bar{\Lambda} \subset \mathbb{R}^3 \backslash \{0\}$, there exists $\lambda_0 \geq 0$ such that $\bar{\Lambda} \subset A_{\lambda_0}$. Then, if we take $\lambda \geq 2 \lambda_0$, we have $\eta_\lambda = 0$ on $\Lambda$. Now, using \eqref{eq:property_g2} and the above remark, we get
\begin{align} \label{eq:eq2}
\int_{\mathbb{R}^3} g_{\varepsilon}(x, |u_n|^2) |u_n|^2 \eta_\lambda & = \int_{\Lambda^c} g_{\varepsilon}(x, |u_n|^2) |u_n|^2 \eta_\lambda \\ \nonumber
& \leq \int_{\Lambda^c} \left( \varepsilon^2 H(x) + \mu V(x) \right) |u_n|^2 \eta_\lambda \\ \nonumber
& \leq \int_{\mathbb{R}^3} \left( \varepsilon^2 H(x) + \mu V(x) \right) |u_n|^2 \eta_\lambda.
\end{align}
Next, using the properties of $\eta_\lambda$ and the magnetic Hardy inequality \eqref{eq:Hardy_inequality_magnetic}, we deduce that
\begin{align} \label{eq:eq3}
\text{Re} \int_{\mathbb{R}^3} i \varepsilon (i \varepsilon \nabla + A) u_n \cdot \nabla \eta_\lambda \overline{u_n} & \leq \varepsilon \left| \int_{\mathbb{R}^3} (i \varepsilon \nabla + A) u_n) \cdot \nabla \eta_\lambda \overline{u_n} \right| \\ \nonumber
& \leq \frac{C \varepsilon}{\lambda} \left( \int_{\mathbb{R}^3} | (i \varepsilon \nabla + A) u_n|^2 \right)^{1/2} \left( \int_{\mathbb{R}^3} \frac{|u_n|^2}{|x|^2} \right)^{1/2} \\ \nonumber
& \leq \frac{4C}{\lambda} \int_{\mathbb{R}^3} | (i \varepsilon \nabla + A) u_n|^2.
\end{align}
Combining \eqref{eq:eq1}, \eqref{eq:eq2} and \eqref{eq:eq3}, we get the estimate
\begin{align*}
\int_{\mathbb{R}^3 \backslash A_{\lambda}} \left( |(i \varepsilon \nabla + A)u_n|^2 + (1 - \mu) V(x) |u_n|^2 \right) & \leq \int_{\mathbb{R}^3} \left( |(i \varepsilon \nabla + A)u_n|^2 + (1 - \mu) V(x) |u_n|^2 \right) \eta_\lambda  \\
& \leq \frac{4C}{\lambda} \| u_n\|_\varepsilon^2 + \varepsilon^2 \int_{\mathbb{R}^3} H(x) |u_n|^2 \eta_\lambda + o(1),
\end{align*}
for $\lambda \geq 2 \lambda_0$. Finally, thanks to \eqref{eq:claim1}, if we take $\lambda > 2 \lambda_\delta$, the second term in the right hand side is smaller than $\delta$. It follows that, for every $\delta > 0$, we can choose (a new) $\lambda_\delta \geq 2 \max\{ \lambda_0, \lambda_\delta\}$ such that \eqref{eq:claim2} holds.

\medbreak

\textit{Step 3.} We are now in a position to deduce the strong convergence. 
We compute
\begin{multline*}
\| u_n - u\|_\varepsilon^2 = \\ \langle J^{\prime}_\varepsilon (u_n) , u_n-u \rangle - \langle J^{\prime}_\varepsilon (u) , u_n-u \rangle + \text{Re} \int_{\mathbb{R}^3} \left[ g_\varepsilon (x,|u_n|^2) u_n - g_\varepsilon (x,|u|^2 ) u \right] (\overline{u_n -u}).
\end{multline*}
From Step 1, we know that $(u_n)_n$ is bounded so that in the right hand side, the first two terms converge to zero. For the last term in the right hand side, we divide the integral in three pieces. We treat separately the integrals on $\Lambda$, $A_{\lambda_\delta} \backslash \Lambda$ and $\mathbb{R}^3 \backslash A_{\lambda_\delta}$ and we next prove that they converge to zero. 

For the integral on $\Lambda$, we can use the fact that $u_n \in H^1_{A,\varepsilon}(\Lambda,\mathbb{C})$ because $\inf_{\Lambda}V > 0$ and $u_n \in \mathcal{X}_\varepsilon$. Then, we conclude by using the compact embedding of $H^1_{A,\varepsilon}(\Lambda,\mathbb{C})$ in $L^{q}(\Lambda,\mathbb{C})$ for $2 \leq q < + \infty$. Indeed, Lemma \ref{lemma:cylindrical-embeddings-compact} applies since $\Lambda$ is bounded away from the axis $x_3$. 

For the integral in $A_{\lambda_\delta} \backslash \Lambda$, we use the fact that $u_n\in H^1(A_{\lambda_\delta} \backslash \Lambda, \mathbb{C})$. Indeed, $A_{\lambda_\delta} \backslash \Lambda$ is bounded. In dimension $3$, this space is compactly embedded in $L^q(A_{\lambda_\delta} \backslash \Lambda,\mathbb{C})$ for $1 \leq q < 6$. Moreover, the penalization $g_\varepsilon$ is bounded on this bounded set. This and the strong convergence in $L^2$ allows to conclude.

The claims in Step 2 were intended to treat the remaining integral. Indeed, using \eqref{eq:claim1} and \eqref{eq:claim2}, we infer that
\begin{align*}
& \limsup_{n\to \infty} \left| \text{Re} \int_{\mathbb{R}^3 \backslash A_{\lambda_\delta}} \left[ g_\varepsilon(x,|u_n|^2) u_n - g_{\varepsilon}(x,|u|^2) u \right] ( \overline{u_n - u} ) \right| \\
& \leq 2 \limsup_{n \to \infty} \int_{\mathbb{R}^3 \backslash A_{\lambda_\delta}} \left( \varepsilon^2 H(x) + \mu V(x) \right) \left( |u_n|^2 + |u|^2 \right) \\
& \leq 4 C (1 + \mu) \delta.
\end{align*}
Then, since $\delta > 0$ is arbitrary, we are done.
\end{proof}

As a direct consequence, we deduce the existence of a least energy solution of the penalized problem \eqref{eq:penalized_problem}. 
\begin{theorem} \label{theorem:existence_solutions_penalized}
Let $g_\varepsilon : \mathbb{R} \times  \mathbb{R}^+ \rightarrow \mathbb{R}$ defined in \eqref{eq:penalized_nonlinearity} satisfy \eqref{eq:property_g1}, \eqref{eq:property_g2}, \eqref{eq:property_g3} and $V \in C(\mathbb{R}^3 \backslash \{0\})$ verify the hypothesis of Section 2.2. Then, for every $\varepsilon > 0$, the functional $\mathcal{J}_{\varepsilon}$ has a non trivial critical point $u_\varepsilon  \in \mathcal{X}_\varepsilon$, which is also a weak solution of \eqref{eq:penalized_problem}, characterized by
\begin{align} \label{eq:minimax_level}
c_\varepsilon = \mathcal{J}_\varepsilon(u_\varepsilon) = \inf_{ u \in \mathcal{X}_\varepsilon \backslash \{0\}} \max_{t >0} \mathcal{J}_{\varepsilon}(tu).
\end{align}
\end{theorem}
This solution $u_\varepsilon $ belongs to $W^{2,q}_{\text{loc}}(\mathbb{R}^3 \backslash \{0\})$ for $2 \leq q < + \infty$ and therefore to $C_{\text{loc}}^{1,\alpha}(\mathbb{R}^3 \backslash \{0\})$. We cannot hope a better regularity since the penalization $g_{\varepsilon}$ is not even continuous.

\medbreak

In the next section, we estimate the critical value $c_\varepsilon$ from above. In the study of the asymptotics of the solutions $u_\varepsilon$, this upper estimate will be useful to determine that the concentration occurs exactly in $\Lambda$. 

\subsection{Upper estimate of the mountain pass level}

\begin{proposition}[Upper estimate of the critical value $c_\varepsilon$] \label{prop:upper_estimate}
Suppose that the assumptions of Theorem \ref{theorem:existence_solutions_penalized} are satisfied. For every $\varepsilon >0$ small enough, the critical value $c_\varepsilon$ defined in \eqref{eq:minimax_level} satisfies
\begin{align} \label{eq:upper_estimate_critical_value}
\liminf_{\varepsilon \to 0} \varepsilon^{-2} c_\varepsilon \leq \inf_{\Lambda \cap \mathcal{H}^\perp} \mathcal{M}.
\end{align} 
Moreover, there exists $C > 0$ such that the solution $u_\varepsilon$ found in Theorem \ref{theorem:existence_solutions_penalized} satisfies
\begin{align} \label{eq:upper_estimate_function}
\| u_\varepsilon \|_\varepsilon^2 \leq C \varepsilon^2.
\end{align}
\end{proposition}

\begin{proof}
Let  $x_0 = (\rho_0 \cos \theta , \rho_0 \sin \theta, 0) \in \Lambda \cap \mathcal{H}^\perp$, with $\rho_0 >0$ and $\theta \in [0,2\pi)$, be such that $\mathcal{M}(x_0) = \inf_{\Lambda \cap \mathcal{H}^\perp} \mathcal{M}$. The existence of $x_0$ is ensured by the continuity of  $\mathcal{M}$ on $\Lambda$ and \eqref{eq:condition_sur_Lambda}. Consider the functional $\mathcal{J}_{a_0}^{A_0}$ defined by \eqref{eq:limit_functional}, with $a_0 = \left[ c(x_0)^2 + V(x_0)\right]$ and $A_0 = (\phi(x_0), 0)$. Next, we define the cut-off function $\eta \in C^\infty_0(\mathbb{R}^+ \times \mathbb{R})$ such that $0 \leq \eta \leq 1$, $\eta = 1$ in a small neighbourhood of $(\rho_0,0)$, and is compactly supported in a small neighbourhood of $(\rho_0,0)$, and $\| \nabla \eta \|_{L^\infty} \leq C$. We define the cylindrically symmetric function 
\begin{align*}
u(x_1, x_2, x_3) = u(\rho, x_3) = \eta(\rho, x_3) v\left(\frac{\rho - \rho_0}{\varepsilon}, \frac{x_3}{\varepsilon}\right),
\end{align*}
where $v$ is the least-energy solution of $\mathcal{J}_{a_0}^{A_0}$. If we perform the change of variables
\begin{align*}
 y_1 = \frac{\rho - \rho_0}{\varepsilon} \quad \text{ and } \quad  y_2 = \frac{x_3}{\varepsilon}
\end{align*}
in the computation of $\mathcal{J}_{\varepsilon}(tu)$, we get
\begin{align*}
\mathcal{J}_{\varepsilon}(tu)& = \frac{t^2}{2} \int_{\mathbb{R}^3} \left[ | (i \varepsilon \nabla + A) u|^2 + V(x) |u|^2 \right] \, \mathrm{d}x - \int_{\mathbb{R}^3} G_{\varepsilon}(x, t^2 |u|^2) \, \mathrm{d}x \\
& = 2 \pi  \varepsilon^2 \frac{t^2}{2} \int_{-\frac{\rho_0}{\varepsilon}}^\infty \int_{\mathbb{R}} \{ \eta^2(\rho_0 + \varepsilon y_1, \varepsilon y_2) \left[ | (i \nabla + \left(\phi(\rho_0 + \varepsilon y_1, \varepsilon y_2), A_3( \rho_0 + \varepsilon y_1, \varepsilon y_2) \right) v|^2 \right] \\
& + \eta^2(\rho_0 + \varepsilon y_1, \varepsilon y_2) \left[ c(\rho_0 + \varepsilon y_1, \varepsilon y_2)^2 + V(\rho_0 + \varepsilon y_1 , \varepsilon y_2) \right] |v|^2 \} \, (\rho_0 + \varepsilon y_1) \, \mathrm{d}y_1 \mathrm{d}y_2 \\
& - 2\pi \varepsilon^2 \int_{-\frac{\rho_0}{\varepsilon}}^\infty \int_{\mathbb{R}} G_\varepsilon(\rho_0 + \varepsilon y_1, \varepsilon y_2, t^2 |v|^2 \eta^2) \, (\rho_0 + \varepsilon y_1) \, \mathrm{d} y_1 \mathrm{d}y_2 + o(\varepsilon^2).
\end{align*}
The term $o(\varepsilon^2)$ includes the terms where the derivatives were applied to $\eta$ instead of $v$. This term is controlled thanks the the compactness of the support of $\eta$ and the control $\| \nabla \eta \|_{L^\infty} \leq C$. Finally, as $\eta$ is compactly supported around $(\rho_0,0)$, $G_\varepsilon(x,t^2|v|^2\eta^2)$ will coincide with $F(t^2|v|^2\eta^2)$ for $\varepsilon$ small enough. We then deduce that 
\begin{align*}
\liminf_{\varepsilon \to 0} \varepsilon^{-2} \mathcal{J}_{\varepsilon}(tu) \leq 2 \pi \rho_0 \, \mathcal{J}_{a_0}^{A_0}(tv) .
\end{align*}
Now we exploit the fact that $c_\varepsilon$ is the least-energy level for $\mathcal{J}_\varepsilon$ and $v$ is the least-energy function for $\mathcal{J}_{a_0}^{A_0}$, as well as Lemma \ref{lemma:least-energy-solution-limit-problem} where $w$ is the least energy solution to $\mathcal{J}_{a_0}^{0}$, to obtain
\begin{align*}
\liminf_{\varepsilon \to 0} \varepsilon^{-2} c_\varepsilon & \leq  \liminf_{\varepsilon \to 0} \varepsilon^{-2} \max_{t > 0} \mathcal{J}_\varepsilon(tu) \leq 2 \pi \rho_0 \, \max_{t>0} \mathcal{J}_{a_0}^{A_0}(tv)  \\
                               & = 2 \pi \rho_0 \, \mathcal{J}_{a_0}^{A_0}(v) = 2 \pi \rho_0 \, \mathcal{J}_{a_0}^{0}(w).                            
\end{align*}
The last equality follows from \eqref{eq:croissance_ground_energy}.

\medbreak

To deduce the second statement of the proposition, we argue as in Step 1 of the proof of Lemma \ref{lemma:existence_solutions_penalized_problem}, with the extra properties that $\mathcal{J}^\prime_{\varepsilon}(u_\varepsilon) = 0$, because we have the additional information that $u_\varepsilon$ is a critical point, and $\mathcal{J}_\varepsilon(u_\varepsilon) = c_\varepsilon \leq C \varepsilon^2$. We then infer that
\begin{align*}
\left( \frac{1}{2} - \frac{1}{\theta} \right) \left( 1 - \max\{4 \kappa, \mu \} \right) \| u_\varepsilon \|_\varepsilon^2 \leq C \varepsilon^2.
\end{align*}
\end{proof}

\section{Asymptotic estimates} \label{section:asymptotic-analysis}

In this section, we study the behaviour of solutions when $\varepsilon \to 0$. With those estimates at hand, we will be able to prove that the solutions of the penalized problem solve the original equation for $\varepsilon$ small enough. 

\subsection{No uniform convergence to $0$ on $\Lambda$}

We start by proving that the solution $u_\varepsilon$ does not converge uniformly to $0$ in $\Lambda$ as $\varepsilon\to 0$. 
\begin{proposition} \label{prop:no_uniform_convergence_in_lambda}
Suppose the assumptions of Theorem \ref{theorem:existence_solutions_penalized} are satisfied and let $(u_\varepsilon)_\varepsilon \subset \mathcal{X}_\varepsilon$ be the solutions found in Theorem \ref{theorem:existence_solutions_penalized}. Then, 
\begin{align*}
\liminf_{\varepsilon\to 0}\|u_\varepsilon\|_{L^\infty(\Lambda)} > 0.
\end{align*} 
\end{proposition}

\begin{proof}
By contradiction, assume that there exists a sequence $(\varepsilon_n)_n \subset \mathbb{R}^+$ such that $\varepsilon_n \to 0$ and $\|u_{\varepsilon_n} \|_{L^\infty(\Lambda)} \rightarrow 0$ as $n \to + \infty$. Using Kato inequality \eqref{eq:magnetic_kato_inequality} and the equation  \eqref{eq:penalized_problem}, we obtain
\begin{align*}
- {\varepsilon_n}^2 \left( \Delta + H \right) |u_{\varepsilon_n}| + (1 - \mu) V |u_{\varepsilon_n}| \leq - \varepsilon_n^2 H |u_{\varepsilon_n}| - \mu V |u_{\varepsilon_n}| + g_{\varepsilon_n}(x, |u_{\varepsilon_n}|^2) |u_{\varepsilon_n}|.
\end{align*}
By \eqref{eq:property_g2}, we infer that the right hand side of the last inequality is non positive in $\Lambda^c$. On the other hand, since we assume that $\|u_{\varepsilon_n}\|_{L^\infty(\Lambda)} \to 0$, the facts that $p > 2$ and $V(x) > 0$ in $\Lambda$ implies $|u_{\varepsilon_n}|^{p-1} \leq \mu V(x) |u_{\varepsilon_n}|$ in $\Lambda$ for $n$ large. We thus conclude that
\begin{align*}
- {\varepsilon_n}^2 \left( \Delta + H(x) | \right)|u_{\varepsilon_n}| + (1 - \mu) V(x) |u_{\varepsilon_n}| \leq 0, \quad \text{ in } \mathbb{R}^3.
\end{align*}
We then reach a contradiction because the comparison principle (Lemma \ref{lemma:comparison}) implies that $|u_{\varepsilon_n}| = 0$ for large $n$. 
\end{proof}

\subsection{Estimates on the rescaled solutions}

As we have seen in Proposition \ref{prop:upper_estimate}, the norm of  the solution $u_\varepsilon$ is of the order $\varepsilon$. It is then natural to rescale $u_\varepsilon$ around some family of points $(\rho_\varepsilon, x_{3,\varepsilon})$ as
\begin{align}\label{eq:rescaled}
v_\varepsilon(y) = u_\varepsilon(x_\varepsilon) =  u_{\varepsilon}(\rho_\varepsilon + \varepsilon y_1, x_{3,\varepsilon} + \varepsilon y_2 ),
\end{align}
where $(x_\varepsilon)_\varepsilon = (\rho_\varepsilon \cos \theta, \rho_\varepsilon \sin \theta, x_{3,\varepsilon})_\varepsilon \subset \bar{\Lambda}$, $\theta \in [0, 2\pi[$. The rescaled solution is defined for  $y=(y_1, y_2) \in \left( -\rho_\varepsilon/\varepsilon, + \infty \right) \times \mathbb{R}$.
The following lemma shows the convergence of those rescaled sequences of solutions.

\begin{lemma}[Convergence of the rescaled solutions] \label{lemma:convergence_rescaled_solutions}
Suppose the assumptions of Theorem \ref{theorem:existence_solutions_penalized} are satisfied. Let $(\varepsilon_n)_n \subset \mathbb{R}^+$ and $(x_n)_n = (\rho_n \cos \theta, \rho_n \sin \theta, x_{3,n})_n \subset \bar{\Lambda}$,  $\theta \in [0, 2 \pi)$ be such that $\varepsilon_n \to 0$ and $x_n \to \bar{x} = (\bar{\rho} \cos \theta, \bar{\rho} \sin \theta, \bar{x}_3) \in \overline{\Lambda} $,  as $n \to + \infty$. Set 
\begin{align*}
& \bar{A} = \left( \phi(\bar{\rho},\bar{x}_3), A_3(\bar{\rho},\bar{x}_3) \right), \qquad  \bar{a} = c^2(\bar{\rho},\bar{x}_3) + V(\bar{\rho},\bar{x}_3).
\end{align*}
Consider the sequence of solutions $(u_{\varepsilon_n})_n \subset \mathcal{X}_{\varepsilon_n}$ found in Theorem \ref{theorem:existence_solutions_penalized}. 
There exists $v \in H^1(\mathbb{R}^2, \mathbb{C})$ such that, up to a subsequence,
\begin{align*}
v_{\varepsilon_n} \rightarrow v \quad \text{ in } \quad C^{1,\alpha}_{\text{loc}}(\mathbb{R}^2, \mathbb{C}) \quad \text{ for } \alpha \in (0,1),
\end{align*}
where $(v_{\varepsilon_n})_n$ is the sequence defined by \eqref{eq:rescaled}, $v$ solves the equation
\begin{align} \label{eq:equation_v}
(i \nabla + \bar{A})^2 v + \bar{a} v = \bar{g}(y, |v|^2)v \quad \text{ in } \mathbb{R}^2,
\end{align}
with
\begin{align}
& \bar{g}(y,|v|^2) = \chi(y) f(|v|^2) + (1 - \chi(y)) \min\{ \mu V(\bar{\rho},\bar{x}_3), f(|v|^2) \},
\end{align}
$\chi$ being the limit a.e. of $\chi_n(y) = \chi_{\Lambda}(\rho_n + \varepsilon_n y_1, x_{3,n} + \varepsilon_n y_2)$. Moreover, we have
\begin{multline} \label{eq:eq1_v}
2 \pi \bar{\rho} \int_{\mathbb{R}^2}  \left( | (i \nabla + \bar{A}) v|^2 + \bar{a} |v|^2 \right) \, \mathrm{d}y  = \\
 \lim_{R \to + \infty} \liminf_{n \to + \infty} \varepsilon_n^{-2} \int_{B_{cyl}(x_n, \varepsilon_n R)} \left( | ( i \varepsilon_n \nabla + A) u_{\varepsilon_n} |^2 + V(x) |u_{\varepsilon_n}|^2 \right) \, \mathrm{d}x .
\end{multline}
\end{lemma}

\begin{proof} We proceed again in several steps. 

\medbreak

\textbf{Step 1: Convergence of the sequence $(v_{\varepsilon_n})_{n}$.} First, the equation solved by $v_{\varepsilon_n}$ is the following
\begin{align} \label{eq:equation_v_n}
(i \nabla + A_n)^2 v_{\varepsilon_n} - \frac{\varepsilon_n}{\rho_n + \varepsilon_n y_1} \frac{\partial v_{\varepsilon_n}}{\partial y_1} + i \frac{\varepsilon_n}{\rho_n + \varepsilon_n y_1} \phi_n v_{\varepsilon_n} + \left[ V_n + c_n^2\right] v_{\varepsilon_n} = g_{\varepsilon_n, n}(y, |v_{\varepsilon_n}|^2) v_{\varepsilon_n}.
\end{align}
The two-dimensional magnetic potential $A_n(y)$ is given by $A_n(y) = \left( \phi_n(y), A_{3,n}(y) \right)$ and the other functions are defined by 
\begin{align*}
\phi_n, A_{3,n},c_n,V_n,g_{\varepsilon_n,n}(y):=\phi,A_3,c,V,g_{\varepsilon_n}(\rho_n + \varepsilon_n y_1, x_{3,n} + \varepsilon_n y_2).
\end{align*} 
By using the definition of $v_{\varepsilon_n}$ and \eqref{eq:upper_estimate_function}, we obtain the following inequality
\begin{align} \label{eq:inequality_v_n}
\int_{-\frac{\rho_n}{\varepsilon_n}}^{+\infty} \int_{\mathbb{R}} \left[ \left| (i \nabla + A_n) v_{\varepsilon_n} \right|^2 + ( V_n + c_n^2 ) |v_{\varepsilon_n}|^2 \right] (\rho_n + \varepsilon_n y_1) \, \mathrm{d}y_1 \mathrm{d}y_2 = \frac{1}{2 \pi \varepsilon_n^2} \| u_{\varepsilon_n} \|_{\varepsilon_n}^2 \leq C,
\end{align}
for $C > 0$ independent from $n$.

Next, we choose e sequence $R_n$ such that $R_n \to + \infty$ and $\varepsilon_n R_n \to 0$ as $n \to + \infty$, and we define the cut-off function $\eta_{R_n} \in C_{c}^{\infty}(\mathbb{R})$ such that $0 \leq \eta_{R_n} \leq 1$,
\begin{align*}
\eta_{R_n} (y) = \left\{ 
\begin{array}{ll} 
0 \quad & \text{ if } |y| \geq R_n \\
1 \quad & \text{ if } |y| \leq R_n/2
\end{array} \right.
\end{align*}
and $\| \nabla \eta_{R_n} \|_{L^\infty} \leq C/R_n$ for some $C > 0$. Since $\overline{\Lambda} \cap \mathcal{H} = \emptyset$, we have that $\rho_n \rightarrow \bar{\rho} >0$, and then, for $n$ sufficiently large, $\varepsilon_n R_n < \bar{\rho}/2 < \rho_n$.
Set $w_n(y) = \eta_{R_n}(y) v_{\varepsilon_n}(y)$, where $v_{\varepsilon_n}$ is extended by $0$ where it is not defined (anyway $\eta_{R_n} =0$ therein). 

We now estimate the $L^2$-norm of $|w_n|$. Observe that if $y_1^2 + y_2^2 \leq R_n^2$, then $(\rho - \rho_n)^2 + (x_3 - x_{3,n})^2 \leq \varepsilon_n^2 R_n^2$, so that for $n$ large enough, $\rho_n + \varepsilon_n y_1\in \overline{\Lambda}$. 
Hence, since by hypothesis $\inf_{\overline{\Lambda}}(c^2 +V) > 0$ and $(\rho_n - \varepsilon_n R_n) > \bar{\rho}/2$ for $n$ large enough, we infer that 
\begin{align} \label{eq:L2-norm-w_n}
\int_{\mathbb{R}^2} |w_n|^2 \, \mathrm{d}y_1 \mathrm{d}y_2 & \leq \int_{B(0,R_n)} |v_{\varepsilon_n}|^2 \, \mathrm{d}y_1 \mathrm{d}y_2 \\ \nonumber
           & \leq  \frac{2}{\bar{\rho}} \, \sup_{\bar{\Lambda}} \frac{1}{c^2 + V} \int_{B(0,R_n)} |v_{\varepsilon_n}|^2 (c_n^2 + V_n) (\rho_n + \varepsilon_n y_1) \, \mathrm{d}y_1 \mathrm{d}y_2. \nonumber 
\end{align}         
Using the fact that $B(0, R_n) \subset \left( - \frac{\rho_n}{\varepsilon_n}, + \infty \right) \times \mathbb{R}$, for $n$ large enough and \eqref{eq:inequality_v_n}, we deduce the estimate
\begin{align*} 
\int_{\mathbb{R}^2} |w_n|^2 \, \mathrm{d}y_1 \mathrm{d}y_2         
         & \leq   \frac{2}{\bar{\rho}} \, \sup_{\bar{\Lambda}} \frac{1}{c^2 + V} \int_{-\frac{\rho_n}{\varepsilon_n}}^{+\infty} \int_{\mathbb{R}} |v_{\varepsilon_n}|^2 (c_n^2 + V_n) (\rho_n + \varepsilon_n y_1) \, \mathrm{d}y_1 \mathrm{d}y_2 \leq C.
\end{align*}   
Next, we study the $L^2$-norm of $\nabla |w_n|$. By using the diamagnetic inequality \eqref{eq:diamagnetic_epsilon} and arguing as before, we get
\begin{align*}
\int_{\mathbb{R}^2} | \nabla |w_n||^2 \, \mathrm{d}y_1 \mathrm{d}y_2 & \leq \int_{\mathbb{R}^2} | (i \nabla + A_n) (\eta_{R_n} v_{\varepsilon_n})|^2 \, \mathrm{d}y_1 \mathrm{d}y_2 \\ 
       & \leq 2 \int_{\mathbb{R}^2} | (i \nabla + A_n) v_{\varepsilon_n}|^2 \eta_{R_n}^2 \, \mathrm{d}y_1 \mathrm{d}y_2 + 2 \int_{\mathbb{R}^2} |\nabla \eta_{R_n}|^2 |v_{\varepsilon_n}|^2 \, \mathrm{d}y_1 \mathrm{d}y_2 \\ 
        & \leq \frac{4}{\bar{\rho}} \, \sup_{\bar{\Lambda}} \frac{1}{c^2 + V} \int_{B(0,R_n)} \left( | (i \nabla + A_n) v_{\varepsilon_n}|^2 + (c_n^2 +V_n) |v_{\varepsilon_n}|^2 \right) (\rho_n + \varepsilon_n y_1) \, \mathrm{d}y_1 \mathrm{d}y_2\\ & \leq C.        
\end{align*}
We have just shown that $(|w_n|)_n\subset H^1(\mathbb{R}^2,\mathbb{R})$ is a bounded sequence. Hence, there exists a function $|v| \in H^1(\mathbb{R}^2,\mathbb{R})$ such that, up to a subsequence, $|w_n|$ converges weakly to $|v|$. Moreover, we deduce from Sobolev embeddings that the convergence is strong in $L^p_{loc}(\mathbb{R}^2,\mathbb{C})$ for $2 \leq p < + \infty$. 

To prove the convergence in $C^{1,\alpha}_{\text{loc}}$, we consider any compact set $K \subset \mathbb{R}^2$. For $n$ sufficiently large, we have $K \subset B(0,\frac{R_n}{2})$ which implies $w_n = v_n$ in $K$. In that compact set $K$, $w_n$ solves the equation \eqref{eq:equation_v_n} and using a standard bootstrap argument (see for example \cite[Theorem 9.1]{GilbargTrudinger}) and the fact that $w_n \in L^p(K,\mathbb{C})$ for $2 \leq p < + \infty$, we conclude that
\begin{align*}
\sup_{n} \| w_n \|_{W^{2,p}(K)} \leq C.
\end{align*}
Finally, since this estimate holds for all $2 \leq p < + \infty$, Sobolev embeddings imply that $w_n = v_n$ converges in $C^{1,\alpha}(K)$ to $v$. The claim then follows from a diagonal procedure. 

\medbreak

\textbf{Step 2: Limit equation satisfied by $v$.} Since $\Lambda$ is smooth, the characteristic functions converge a.e. to a measurable function $0 \leq \chi(y) \leq 1$. We therefore obtain equation \eqref{eq:equation_v} from \eqref{eq:equation_v_n}  by using the $C^{1,\alpha}_{\text{loc}}$-convergence. Moreover, if $\bar{x} \in \Lambda$, we remark that $\bar{g}(y,|v|^2) = f(|v|^2)$, that is $\chi \equiv 1$. 

\medbreak

\textbf{Step 3: Proof of the estimate \eqref{eq:eq1_v}.} Using the preceding arguments and the $C^{1,\alpha}_{loc}$-convergence, we have 
\begin{align*}
\liminf_{n \to + \infty} \varepsilon_n^{-2}  \int_{B_{cyl}(x_n, \varepsilon_n R)} & \left( | (i \varepsilon_n \nabla + A) u_{\varepsilon_n} |^2 + V(x) |u_{\varepsilon_n}|^2 \right) \, \mathrm{d}x \\
& = 2 \pi \liminf_{n \to + \infty} \int_{B(0,R)} \left[ | (i \nabla + A_n) v_{\varepsilon_n}|^2 + ( c_n^2 + V_n ) |v_{\varepsilon_n}|^2 \right] (\rho_n + \varepsilon_n y_1) \, \mathrm{d}y_1 \mathrm{d}y_2 \\
& = 2 \pi \bar{\rho} \int_{B(0,R)} \left[ | (i \nabla + \bar{A}) v|^2 + \bar{a} |v|^2 \right]  \, \mathrm{d}y_1 \mathrm{d}y_2.
\end{align*}
Finally, we let $R$ go to $+\infty$ to complete the proof. 
\end{proof}

Next, we examine the contribution of $u_\varepsilon$ to the action functional in a neighbourhood of a circle. In particular, we derive a lower estimate on the action of $u_\varepsilon$ which accounts for the number of circles around which $u_\varepsilon$ is non negligible. By combining the next lemmas with the upper estimate on the critical level $c_\varepsilon$, we reach the conclusion that $u_\varepsilon$ concentrates around exactly one circle.

\begin{lemma}[lower bound in a small ball]
Suppose that the assumptions of Theorem \ref{theorem:existence_solutions_penalized} are satisfied. Let $(\varepsilon_n)_n \subset \mathbb{R}^+$ and $(x_n)_n = (\rho_n \cos \theta, \rho_n \sin \theta, x_{3,n})_n \subset \bar{\Lambda}$ be such that $\varepsilon_n \rightarrow 0$ and $x_n \to \bar{x} = (\bar{\rho} \cos \theta, \bar{\rho} \sin \theta, \bar{x}_3) \in \bar{\Lambda}$ as $n \to + \infty$, $\theta \in [0,2\pi)$. Let $(u_{\varepsilon_n})_n \subset \mathcal{X}_{\varepsilon_n}$ be the solutions found in Theorem \ref{theorem:existence_solutions_penalized}. If
\begin{align} \label{eq:maximum_point}
\liminf_{n \to + \infty} |u_{\varepsilon_n}(x_n)| > 0,
\end{align}
then, up to a subsequence, we have
\begin{align*}
\liminf_{R \to + \infty} \liminf_{n \to + \infty} \varepsilon_n^{-2} \int_{B_{cyl}(x_n, \varepsilon_n R)} \frac{1}{2} ( | ( i \varepsilon_n \nabla + A) u_{\varepsilon_n} |^2 + V(x) |u_{\varepsilon_n}|^2 ) - G_{\varepsilon_n}(x,|u_{\varepsilon_n}|^2)\nonumber \\ \geq  \mathcal{M}(\bar{\rho}, \bar{x}_3). 
\end{align*}
\end{lemma}

\begin{proof} We set again $v_{\varepsilon_n}$ as in \eqref{eq:rescaled}.
First, by \eqref{eq:maximum_point}, $|v(0)| = \lim_{n \to + \infty} |v_{\varepsilon_n}(0)| > 0$, then $v$ is not identically zero. Moreover, we know from Lemma \ref{eq:rescaled} that $v$ satisfies the equation \eqref{eq:equation_v}. This implies that $v$ is a critical point of the functional $\mathcal{G}^{\bar{A}}_{\bar{a}} : H^1(\mathbb{R}^2, \mathbb{C}) \rightarrow \mathbb{R}$ defined by
\begin{align*}
\mathcal{G}^{\bar{A}}_{\bar{a}}(u) = \frac{1}{2} \int_{\mathbb{R}^2} | (i \nabla + \bar{A}) u|^2 + \bar{a} |u|^2 \, \mathrm{d}y - \int_{\mathbb{R}^2} \bar{G}(y, |u|^2) \, \mathrm{d}y,
\end{align*}
$\bar{a}$ and $\bar{A}$ being defined in Lemma \ref{lemma:convergence_rescaled_solutions}, and where
\begin{align*}
\bar{G}(y,s) = \frac{1}{2} \int_{0}^s \bar{g}(y, \sigma) \, \mathrm{d}\sigma.
\end{align*}
Since $\bar{g}(y, |u|^2) \leq f(|u|^2)$, it follows immediately that
\begin{align*}
\mathcal{G}^{\bar{A}}_{\bar{a}}(u) \geq \mathcal{J}^{\bar{A}}_{\bar{a}}(u).
\end{align*}
Since $v$ is a critical point of $\mathcal{G}^{\bar{A}}_{\bar{a}}$ and $\bar{g}$ satisfies the property \eqref{eq:property_g3}, we have that
\begin{align*}
\mathcal{G}^{\bar{A}}_{\bar{a}}(v) = \sup_{t >0} \mathcal{G}^{\bar{A}}_{\bar{a}}(tv) & \geq \inf_{u \in H^1(\mathbb{R}^2,\mathbb{C})} \sup_{t >0} \mathcal{G}^{\bar{A}}_{\bar{a}}(tu) \\
   & \geq \inf_{u \in H^1(\mathbb{R}^2,\mathbb{C})} \sup_{t >0} \mathcal{J}^{\bar{A}}_{\bar{a}}(tu) = \mathcal{E}(\bar{A},\bar{a}) = \mathcal{E}(0,1) \, \bar{a}^{\frac{2}{p-2}}.
\end{align*}
By using the $C^{1,\alpha}_{\text{loc}}$-convergence of the sequence $(v_{\varepsilon_n})_n$, we obtain that
\begin{align*}
& \liminf_{n \to + \infty} \varepsilon_n^{-2} \int_{B_{cyl}(x_n, \varepsilon_n R)} \left[ \frac{1}{2} \left( | ( i \varepsilon_n \nabla + A) u_{\varepsilon_n} |^2 + V(x) |u_{\varepsilon_n}|^2 \right) - G_{\varepsilon_n}(x,|u_{\varepsilon_n}|^2) \right] \, \mathrm{d}x = \\
& 2 \pi  \liminf_{n \to + \infty} \int_{B(0,R)} \left[ \frac{1}{2} \left( | ( i \nabla + A_n) v_{\varepsilon_n} |^2 + ( V_n  + c_n^2 ) |v_{\varepsilon_n}|^2 \right) - G_{\varepsilon_n,n}(y,|v_{\varepsilon_n}|^2) \right] (\rho_n + \varepsilon_n y_1) \, \mathrm{d}y =  \\
& 2 \pi  \bar{\rho} \int_{B(0,R)} \left[ \frac{1}{2} \left( | ( i \nabla + \bar{A}) v |^2 +  \bar{a} |v|^2 \right) - \bar{G}(y,|v|^2) \right] \, \mathrm{d}y .
\end{align*}
Finally, we let $R \to + \infty$ to conclude.

\end{proof}

The following lemma estimates what happens outside the small balls where $u_{\varepsilon}$ concentrates.  In particular we show that the contribution to the action of $u_{\varepsilon}$ is  nonnegative so that the lower estimate from the preceding lemma is meaningful.

\begin{lemma}[Inferior bound outside small balls]
Assume that the assumptions of Theorem \ref{theorem:existence_solutions_penalized} are satisfied. Let $(\varepsilon_n)_n \subset \mathbb{R}^+$ and $(x_n^i)_n = (\rho_n^i \cos \theta, \rho_n^i \sin \theta, x_{3,n}^i)_n \subset \bar{\Lambda}$ be such that $\varepsilon_n \rightarrow 0$ and $x_n^i \to \bar{x}^i = (\bar{\rho}^i \cos \theta, \bar{\rho}^i \sin \theta, \bar{x}_3^i) \in \bar{\Lambda}$, for $1 \leq i \leq M$, as $n \to + \infty$, $\theta \in [0,2\pi)$. Let $(u_{\varepsilon_n})_n \subset \mathcal{X}_{\varepsilon_n}$ be the solutions found in Theorem \ref{theorem:existence_solutions_penalized}.  Then, up to a subsequence, we have
\begin{align} \label{eq:lower_estimate_2}
\liminf_{R \to + \infty} \liminf_{n \to + \infty} \varepsilon_n^{-2} \int_{\mathbb{R}^3 \backslash \mathcal{B}_n(R)} \frac{1}{2} \left( | ( i \varepsilon_n \nabla + A) u_{\varepsilon_n} |^2 + V(x) |u_{\varepsilon_n}|^2 \right) - G_{\varepsilon_n}(x,|u_{\varepsilon_n}|^2) \geq 0,
\end{align}
where 
\begin{equation}\label{eq:defBnR}
\mathcal{B}_n(R) = \cup_{i = 1}^M B_{cyl}(x_n^i, \varepsilon_n R).
\end{equation}
\end{lemma}

\begin{proof}
We consider yet another smooth test function $\eta_{R,\varepsilon_n}$ such that $\eta_{R,\varepsilon_n} = 0$ on $\mathcal{B}_n(R/2)$, $\eta_{R,\varepsilon_n}= 1$ on $\mathbb{R}^2 \backslash \mathcal{B}_n(R)$ and $\| \nabla \eta_{R,\varepsilon_n} \|_{L^\infty} \leq C/(\varepsilon_n R)$. From \eqref{eq:property_g1} and \eqref{eq:property_g2}, we infer that
\begin{align*}
\int_{\mathbb{R}^3 \backslash \mathcal{B}_n(R)} & \left[ \frac{1}{2} \left( | ( i \varepsilon_n \nabla + A) u_{\varepsilon_n} |^2 + V(x) |u_{\varepsilon_n}|^2 \right) - G_{\varepsilon_n}(x,|u_{\varepsilon_n}|^2) \right] \, \mathrm{d}x \\
& \geq \int_{\mathbb{R}^3 \backslash \mathcal{B}_n(R)} \left[ \frac{1}{2} \left( | ( i \varepsilon_n \nabla + A) u_{\varepsilon_n} |^2 + V(x) |u_{\varepsilon_n}|^2 \right) - g_{\varepsilon_n}(x,|u_{\varepsilon_n}|^2)|u_{\varepsilon_n}|^2 \right] \, \mathrm{d}x.
\end{align*}
If we test the equation  \eqref{eq:penalized_problem} on $(u_{\varepsilon_n} \eta_{R,\varepsilon_n})$, we obtain
\begin{align*}
0 & = \int_{\mathbb{R}^3 \backslash \mathcal{B}_n(R)} \left[| (i \varepsilon_n \nabla + A) u_{\varepsilon_n}|^2 + V(x) |u_{\varepsilon_n}|^2 - g_{\varepsilon_n}(x, |u_{\varepsilon_n}|^2)|u_{\varepsilon_n}|^2 \right] \, \mathrm{d}x \\
& + \int_{\mathcal{B}_n(R) \backslash \mathcal{B}_n(R/2)} \left[ | (i \varepsilon_n \nabla + A) u_{\varepsilon_n}|^2 + V(x) |u_{\varepsilon_n}|^2 - g_{\varepsilon_n}(x, |u_{\varepsilon_n}|^2)|u_{\varepsilon_n}|^2 \right] \eta_{R,\varepsilon_n}^2 \, \mathrm{d}x \\
& - i \varepsilon_n \int_{\mathcal{B}_n(R) \backslash \mathcal{B}_n(R/2)} \nabla \eta_{R,\varepsilon_n} \cdot (i \varepsilon_n \nabla + A) u_{\varepsilon_n} \overline{u_{\varepsilon_n}} \, \mathrm{d}x.
\end{align*}
Then, to deduce the estimate \eqref{eq:lower_estimate_2}, it is enough to estimate the last two integrals in the annular region $\mathcal{A}_n(R) = \mathcal{B}_n(R) \backslash \mathcal{B}_n(R/2)$.

We start with the first of these two terms. Thanks to the fact that $\mathcal{A}_n(R)$ is a bounded set having the cylindrical symmetry and such that $\overline{\mathcal{A}_n(R)} \cap \mathcal{H}=\emptyset$, we can use the compact embeddings from Lemma \ref{lemma:cylindrical-embeddings-compact}. Then, we conclude that
\begin{align*}
& \liminf_{n \to +\infty} \varepsilon_n^{-2} \left| \int_{\mathcal{A}_n(R)} \left[ | (i \varepsilon_n \nabla + A) u_{\varepsilon_n}|^2 + V(x) |u_{\varepsilon_n}|^2 - g_{\varepsilon_n}(x, |u_{\varepsilon_n}|^2)|u_{\varepsilon_n}|^2 \right] \eta_{R,\varepsilon_n}^2 \, \mathrm{d}x \right| \\
& \qquad \qquad \qquad \qquad \leq \liminf_{n \to + \infty}  \varepsilon_n^{-2} C \left( I_{n,R}^2 + I_{n,R}^q \right),
\end{align*}
where we denoted
\begin{align*}
I_{n,R} = \left[ \int_{\mathcal{A}_n(R)} \left[ | ( i \varepsilon_n + A) u_{\varepsilon_n}|^2 + V(x)  |u_{\varepsilon_n}|^2  \right] \, \mathrm{d}x \right]^{\frac{1}{2}}.
\end{align*}
Next, we estimate the second term
\begin{align*}
& \liminf_{n \to + \infty} \varepsilon_n^{-2} \left| \varepsilon_n \int_{\mathcal{A}_n(R)} (i \varepsilon_n \nabla + A) u_{\varepsilon_n} \cdot \nabla \eta_{R,\varepsilon_n} \overline{u_{\varepsilon_n}} \, \mathrm{d}x \right| \\ 
& \qquad \leq \liminf_{n \to + \infty} C \varepsilon_n^{-2}R^{-1} \int_{\mathcal{A}_n(R)} | (i \varepsilon_n \nabla + A ) u_{\varepsilon_n}| |u_{\varepsilon_n}| \, \mathrm{d}x \\
& \qquad \leq \liminf_{n \to + \infty} 2C \varepsilon_n^{-2}R^{-1} \int_{\mathcal{A}_n(R)} \left( |(i \varepsilon_n \nabla + A) u_{\varepsilon_n}|^2 + |u_{\varepsilon_n}|^2 \right) \, \mathrm{d}x \\
& \qquad \leq \liminf_{n \to + \infty} 2C \varepsilon_n^{-2}R^{-1} I^2_{n,R}.
\end{align*}
Finally, by taking the $\liminf_{R \to +\infty}$ and using relation \eqref{eq:eq1_v}, we obtain that both integrals converge to zero, which concludes the result.
\end{proof}

The next lemma combines the informations from the two preceding ones and yields a lower bound on the action of $u_\varepsilon$ as a function of the points in $\bar\Lambda$ where the solution concentrates.

\begin{lemma}[lower bound on the critical level] \label{lemma:inferior_bound_several_balls}
Suppose that the assumptions of Theorem \ref{theorem:existence_solutions_penalized} are satisfied. Let $(\varepsilon_n)_n \subset \mathbb{R}^+$ and $(x_n^i)_n = (\rho_n^i \cos \theta, \rho_n^i \sin \theta, x_{3,n}^i)_n \subset \bar{\Lambda}$ be such that $\varepsilon_n \rightarrow 0$ and $x_n^i \to \bar{x}^i = (\bar{\rho}^i \cos \theta , \bar{\rho}^i \sin \theta, x_{3}^i) \in \bar{\Lambda}$, for $1 \leq i \leq M$, as $n \to + \infty$, $\theta \in [0,2\pi)$. Let $(u_{\varepsilon_n})_n \subset \mathcal{X}_{\varepsilon_n}$ be the solutions found in Theorem \ref{theorem:existence_solutions_penalized}. Assume that for every $1 \leq i < j \leq M$, we have
\begin{align} \label{eq:distinct_ball}
\limsup_{n \to + \infty} \frac{\mathrm{d}_{cyl}(x_n^i,x_n^j)}{\varepsilon_n} = + \infty,
\end{align}
and
\begin{align*}
\liminf_{n \to + \infty} | u_{\varepsilon_n}(x_n^i)| > 0.
\end{align*}
Then it holds
\begin{align*}
\liminf_{n \to + \infty} \varepsilon_n^{-2} c_{\varepsilon_n} \geq \sum_{i = 1}^M \mathcal{M}(\bar{\rho}^i, \bar{x}_{3}^i).
\end{align*} 
\end{lemma}

\begin{proof}
We infer from the previous lemmas that for every $\delta > 0$, there exists $R_{\delta} >0$ large enough, such that for all $R > R_\delta$
\begin{align*}
& \liminf_{n \to + \infty} \varepsilon_n^{-2} \int_{\mathbb{R}^3 \backslash \mathcal{B}_n(R)} \left[ \frac{1}{2} ( | (i \varepsilon_n \nabla + A) u_{\varepsilon_n} |^2 + V(x) |u_{\varepsilon_n}|^2 ) - G_{\varepsilon_n}(x,|u_{\varepsilon_n}|^2) \right] \, \mathrm{d}x \geq - \delta \\
&  \liminf_{n \to + \infty} \varepsilon_n^{-2} \int_{B_{cyl}(x_n^i, \varepsilon_n R)} \left[ \frac{1}{2} ( | (i \varepsilon_n \nabla + A) u_{\varepsilon_n}|^2 + V(x) |u_{\varepsilon_n}|^2 ) - G_{\varepsilon_n}(x,|u_{\varepsilon_n}|^2) \right] \, \mathrm{d}x \geq  \mathcal{M}(\bar{\rho}^i, \bar{x}_3^i) - \delta,
\end{align*}
where $\mathcal{B}_n(R)$ is defined in \eqref{eq:defBnR}. 
Then, thanks to the hypothesis \eqref{eq:distinct_ball}, the balls are disjoint. We then decompose $\varepsilon_n^{-2} \mathcal{J}(u_{\varepsilon_n})$ as the sum of the $M$ integrals on each ball $B_{cyl}(x_n^i, \varepsilon_n R)$ and one integral in $\mathbb{R}^3 \backslash \mathcal{B}_n(R)$. We then have
\begin{align*}
\liminf_{n \to + \infty} \varepsilon_n^{-2} \mathcal{J}(u_{\varepsilon_n}) \geq  \sum_{i=1}^M \mathcal{M}(\bar{\rho}^i, \bar{x}_3^i) - (M+1) \delta. 
\end{align*}
Since $\delta >0$ is arbitrary, the conclusion follows.
\end{proof}

The following proposition is a key result of the proof. It concludes to the existence of a sequence of maximum points for $u_{\varepsilon}$ in $\bar{\Lambda}$ and tells us that that sequence of maximum points will in fact converge to the point of infimum of our concentration function $\mathcal{M}$ at the interior of $\Lambda$. 

\begin{proposition} \label{prop:uniform_convergence_outside_small_balls}
Suppose that the assumptions of Theorem \ref{theorem:existence_solutions_penalized} are satisfied. Let $(u_{\varepsilon})_\varepsilon \subset \mathcal{X}_\varepsilon$ be the solutions found in Theorem \ref{theorem:existence_solutions_penalized} for $\varepsilon > 0$. Then, there exist $(x_\varepsilon)_\varepsilon = (\rho_\varepsilon \cos \theta, \rho_\varepsilon \sin \theta, x_{3,\varepsilon})_\varepsilon \subset \bar{\Lambda}$ such that
\begin{align} \label{eq:maximum_u}
\liminf_{\varepsilon \to 0} | u_{\varepsilon}(x_\varepsilon)| > 0.
\end{align}
Moreover, we have
\begin{itemize}
\item[(i)] $\displaystyle \limsup_{\varepsilon \to 0} \frac{\mathrm{d}_{cyl}(x_\varepsilon, \mathcal{H}^\perp)}{\varepsilon} < + \infty$, that is $x_{3,\varepsilon} \to 0$;
\item[(ii)] $ \displaystyle \liminf_{\varepsilon \to 0} \mathrm{d}_{cyl}(x_\varepsilon, \partial \Lambda)> 0$;
\item[(iii)] $\displaystyle \lim_{\varepsilon \to 0} \mathcal{M}(x_\varepsilon) = \inf_{\Lambda \cap \mathcal{H}^{\perp}} \mathcal{M}$;
\item[(iv)] for every $\delta > 0$, there exists $R_\delta > 0$, such that for every $R > R_\delta$ there exist $\varepsilon_R >0$ such that, for every $\varepsilon < \varepsilon_R$, $| u_{\varepsilon}| < \delta$ in $\Lambda \backslash B_{cyl}(x_\varepsilon,\varepsilon R)$.
\end{itemize}
\end{proposition}

\begin{proof}
First, observe that the existence of a sequence $(x_\varepsilon)_\varepsilon \subset \bar{\Lambda}$ of local maximum points of $|u_{\varepsilon}|$ in $\bar{\Lambda}$ follows from the continuity of $u_{\varepsilon}$. The estimate \eqref{eq:maximum_u} holds because we know from Proposition \ref{prop:no_uniform_convergence_in_lambda} that $u_{\varepsilon}$ does not converge uniformly to zero in $\bar{\Lambda}$. 

\medbreak 
 
\textbf{Proof of assertion} (i). By contradiction, assume that there exist sequences $(\varepsilon_n)_n \subset \mathbb{R}^+$ and $(x_n)_n \subset \bar{\Lambda}$ such that $\varepsilon_n \to 0$ and $x_n \to \bar{x} = (\bar{\rho} \cos \theta, \bar{\rho} \sin \theta, \bar{x}_3) \in \bar{\Lambda}$, $\theta \in [0, 2 \pi)$ (this is always possible because of the compactness of $\bar{\Lambda}$),
\begin{align*}
\liminf_{n \to + \infty} | u_{\varepsilon_n}(x_n)| > 0,
\end{align*} 
and
\begin{align*}
\limsup_{n \to + \infty} \frac{\mathrm{d}(x_n,\mathcal{H}^\perp)}{\varepsilon_n} = + \infty.
\end{align*}
Let $g_{ref} \in G$ be the reflection with respect to $\mathcal{H}^{\perp}$. We know that $u_{\varepsilon_n} \circ g_{ref} = u_{\varepsilon_n}$, so that
\begin{align*}
\liminf_{n \to + \infty} | u_{\varepsilon_n} (g_{ref}(x_n)) | > 0.
\end{align*}
Moreover, by our assumption 
\begin{align*}
\lim_{n \to + \infty} \frac{\mathrm{d}_{cyl}(g_{ref}(x_n),\mathcal{H}^\perp)}{\varepsilon_n} = + \infty.
\end{align*}
Therefore, we infer that
\begin{align*}
\limsup_{n \to + \infty} \frac{\mathrm{d}_{cyl}(x_n , g_{ref}(x_n))}{\varepsilon_n} = + \infty,
\end{align*} 
We can now use Lemma \ref{lemma:inferior_bound_several_balls} to deduce that
\begin{align*}
\liminf_{n \to + \infty} \varepsilon_n^{-2} c_{\varepsilon_n} \geq  \left( \mathcal{M}(\bar{x}) + \mathcal{M}(g_{ref}(\bar{x})) \right) \geq 2 \inf_{\Lambda} \mathcal{M},
\end{align*}
whereas we know from \eqref{eq:upper_estimate_critical_value} in Proposition \ref{prop:upper_estimate} that
\begin{align*}
\liminf_{n \to + \infty} \varepsilon_n^{-2} c_{\varepsilon_n}\le  \inf_{\Lambda \cap \mathcal{H}^{\perp}} \mathcal{M}.
\end{align*}
This yields the inequality
\begin{align*}
2 \inf_{\Lambda} \mathcal{M} \leq \inf_{\Lambda \cap \mathcal{H}^{\perp}} \mathcal{M},
\end{align*}
which is impossible because of the property \eqref{eq:condition_sur_Lambda} of the set $\Lambda$.


\medbreak

\textbf{Proof of assertion} (ii). Arguing again by contradiction, assume that there exist sequences $(\varepsilon_n)_n \subset \mathbb{R}^+$ and $(x_n)_n \subset \bar{\Lambda}$ such that $\varepsilon_n \to 0$, 
\begin{align*}
\liminf_{n \to + \infty} | u_{\varepsilon_n}(x_n)| > 0,
\end{align*} 
and
\begin{align*}
\lim_{n \to + \infty} \mathrm{d}_{cyl}(x_n, \partial \Lambda)=0,
\end{align*}
that is $x_n \to \bar{x} = (\bar{\rho} \cos \theta, \bar{\rho} \sin \theta, \bar{x}_3) \in \partial \Lambda$, $\theta \in [0, 2 \pi)$. By assertion $(i)$, we also know that $\bar{x} \in \mathcal{H}^\perp$. From Lemma \ref{lemma:inferior_bound_several_balls}, we have
\begin{align*}
\liminf_{n \to + \infty} \varepsilon_n^{-2} c_{\varepsilon_n} \geq  \mathcal{M}(\bar{\rho},\bar{x}_3) \geq  \inf_{\partial \Lambda \cap \mathcal{H}^{\perp}} \mathcal{M},
\end{align*}
so that \eqref{eq:upper_estimate_critical_value} in Proposition \ref{prop:upper_estimate} implies
\begin{align*}
\inf_{\partial \Lambda \cap \mathcal{H}^\perp} \mathcal{M} \leq \inf_{\Lambda \cap \mathcal{H}^\perp} \mathcal{M},
\end{align*}
which is again a contradiction to \eqref{eq:condition_sur_Lambda}. 

\medbreak

\textbf{Proof of assertion} (iii). 
This is also an easy consequence of Proposition \ref{prop:upper_estimate} and Lemma \ref{lemma:inferior_bound_several_balls}. Indeed, using also $(i)$, we can still assume the existence of a sequence $(x_n)_n$ such that $x_n$ converges to some $\bar{x}= (\bar{\rho} \cos \theta, \bar{\rho} \sin \theta, 0) \in \bar\Lambda \cap \mathcal{H}^\perp$. Then combining 
Lemma \ref{lemma:inferior_bound_several_balls} and Proposition \ref{prop:upper_estimate}, we deduce that
\begin{align*}
 \mathcal{M}(\bar{x})\le \liminf_{n \to + \infty} \varepsilon_{n}^{-2} c_{\varepsilon_n} \le  \inf_{\Lambda \cap \mathcal{H}^{\perp}} \mathcal{M}.
\end{align*}
Assume $\bar{x} \in \partial\Lambda \cap \mathcal{H}^\perp$. Then, by \eqref{eq:condition_sur_Lambda} and the last inequality, we have 
$$ \inf_{\Lambda \cap \mathcal{H}^{\perp}} \mathcal{M} < \inf_{\partial \Lambda \cap \mathcal{H}^\perp} \mathcal{M} \leq \mathcal{M}(\bar{x})\le \inf_{\Lambda \cap \mathcal{H}^{\perp}} \mathcal{M},$$
which is a contradiction. Henceforth, we deduce that $\bar x\in \Lambda \cap \mathcal{H}^\perp$ and $\lim_{n \to + \infty} \mathcal{M}(x_n)=\mathcal{M}(\bar{x})= \inf_{\Lambda \cap \mathcal{H}^{\perp}} \mathcal{M}$.

\medbreak

\textbf{Proof of assertion} (iv). 
Assume by contradiction the existence of $\delta >0$ and a sequence $y_n \in \bar{\Lambda}$ such that
\begin{align*}
| u_{\varepsilon_n}(y_n)|> \delta,
\end{align*}
and
\begin{align*}
\lim_{n \to + \infty} \frac{\mathrm{d}_{cyl}(x_n,y_n)}{\varepsilon_n} = + \infty.
\end{align*}
Up to a subsequence, we know that $y_n \rightarrow \bar{y} \in \bar \Lambda \cap \mathcal{H}^\perp$, Then, using again Lemma \ref{lemma:inferior_bound_several_balls}, Proposition \ref{prop:upper_estimate} and \eqref{eq:condition_sur_Lambda}, we obtain
\begin{align*}
\inf_{\Lambda \cap \mathcal{H}^\perp} \mathcal{M} \geq \liminf_{n \to + \infty} \varepsilon_n^{-2} c_{\varepsilon_n} \geq \left( \mathcal{M}(\bar{x}) + \mathcal{M}(\bar{y}) \right) \geq 2 \inf_{\Lambda\cap \mathcal{H}^\perp} \mathcal{M},
\end{align*}
which is impossible.
\end{proof}

\section{Solutions of the initial problem} \label{section:solution-initial-problem}

All this section is inspired by \cite{B-DC-VS}, where they study concentration of solutions around $k$-spheres for Laplacian problems. 

\subsection{Linear inequation outside small balls}

\begin{lemma} \label{lemma:inequation_outside_small_balls}
Suppose that the assumptions of Theorem \ref{theorem:existence_solutions_penalized} are satisfied. Let $(u_{\varepsilon})_\varepsilon \subset \mathcal{X}_\varepsilon$ be the solutions found in Theorem \ref{theorem:existence_solutions_penalized}. Let $(x_\varepsilon)_\varepsilon \subset \bar{\Lambda}$, found in Proposition \ref{prop:uniform_convergence_outside_small_balls}, be such that
\begin{align*}
\liminf_{\varepsilon \to 0} | u_{\varepsilon}(x_\varepsilon)| > 0.
\end{align*}
Then, there exists $r_0 > 0$ such that for every $r > r_0$, there exists $\varepsilon_r > 0$ such that for every $\varepsilon < \varepsilon_r$,
\begin{align*}
- \varepsilon^2 \left( \Delta + H \right) |u_{\varepsilon}| + (1 - \mu) V |u_{\varepsilon}| \leq 0 \quad \text{ in } \mathbb{R}^3 \backslash B_{cyl}(x_\varepsilon,\varepsilon r). 
\end{align*}
\end{lemma}

\begin{proof}
First, we have that
\begin{align*}
\mu V(x) \geq \delta > 0,
\end{align*}
for $x \in \Lambda$. By Proposition \ref{prop:uniform_convergence_outside_small_balls} (iv), there exists $r_0 > 0$ sufficiently large, such that, for every $r > r_0$ there exist $\varepsilon_r >0$ such that for every $\varepsilon < \varepsilon_r$,
\begin{align*}
|u_{\varepsilon}(x)|^{p-2} < \delta \quad \text{ in } \Lambda \backslash B_{cyl}(x_\varepsilon,\varepsilon r).
\end{align*}
Then, we use the Kato inequality \eqref{eq:kato_inequality} to obtain
\begin{align*}
- \varepsilon^2 \left( \Delta + H \right) |u_{\varepsilon}| + (1 - \mu) V |u_{\varepsilon}| \leq |u_{\varepsilon}|^{p-1} - \mu V |u_{\varepsilon}| - \varepsilon^2 H |u_{\varepsilon}| < 0 \quad \text{ in } \Lambda \backslash B_{cyl}(x_\varepsilon, \varepsilon r).
\end{align*}
Now, in $\mathbb{R}^3 \backslash \Lambda$, we use again the Kato inequality to obtain
\begin{align*}
- \varepsilon^2 \left( \Delta + H \right) |u_{\varepsilon}| + (1 - \mu) V |u_{\varepsilon}| \leq 0 \quad \text{ in } \mathbb{R}^3 \backslash \Lambda,
\end{align*}
by the definition of the nonlinearity $g_{\varepsilon}$ in $\mathbb{R}^3 \backslash \Lambda$. This concludes the proof.
\end{proof}

\subsection{Barrier functions}

Once we can construct functions $w_{\varepsilon}$ verifying the opposite inequation 
\begin{align*}
-\varepsilon^2 \left( \Delta + H \right) w_{\varepsilon} + (1-\mu) V w_{\varepsilon} \geq 0 \quad \text{ in } \mathbb{R}^3 \backslash B_{cyl}(x_\varepsilon, \varepsilon r)
\end{align*}
with some convenient boundary conditions on $\partial B_{cyl}(x_\varepsilon, \varepsilon r)$, Lemma \ref{lemma:inequation_outside_small_balls} suggests that we can use the comparison principle to obtain an upper bound on $|u_{\varepsilon}|$. Those functions $w_{\varepsilon}$ will be chosen in such a good way that the bound $|u_{\varepsilon}|\leq C w_{\varepsilon}$ imply that $|u_{\varepsilon}|^{p-2} \leq \mu V(x) + \varepsilon^2 H(x)$ for all $x \in \mathbb{R}^3 \backslash \Lambda$, so that we recover solutions of the initial problem \eqref{eq:initialproblem}. 

\medbreak

We now define more precisely the notion of barrier functions.

\begin{definition} \label{def:barrier_function}
Let $(x_\varepsilon)_\varepsilon \subset \mathbb{R}^3$ and $r >0$. We say that $(w_{\varepsilon})_\varepsilon \subset C^{1,\alpha}(\mathbb{R}^3 \backslash \{ 0 \} \backslash B_{cyl}(x_\varepsilon,\varepsilon r))$ is a family of barrier functions if there exists $\varepsilon_0 > 0$ such that, for every $\varepsilon < \varepsilon_0$, we have that
\begin{itemize}
\item[(i)] $w_{\varepsilon}$ satisfies the inequation
\begin{align*}
- \varepsilon^2 \left( \Delta + H \right) w_{\varepsilon} + (1 - \mu) V w_{\varepsilon} \geq 0 \quad \text{ in } \mathbb{R}^3 \backslash B_{cyl}(x_\varepsilon, \varepsilon r);
\end{align*}
\item[(ii)] $\nabla w_{\varepsilon} \in L^2(\mathbb{R}^3 \backslash B_{cyl}(x_\varepsilon,\varepsilon r))$;
\item[(iii)] $w_{\varepsilon} \geq 1$ on $\partial B_{cyl}(x_\varepsilon, \varepsilon r)$.
\end{itemize}
\end{definition}

\subsubsection{Construction of the comparison functions}

In this section, we recall how to construct some comparison functions in $\Lambda$ and in $\mathbb{R}^3 \backslash \Lambda$. Those comparison functions will be used to construct the barrier functions. We first begin by the construction in $\mathbb{R}^3 \backslash \Lambda$.

\begin{lemma} \label{lemma:Psi}
For every $\varepsilon > 0$, there exists $\Psi_{\varepsilon} \in C^{1,\alpha}_{\text{loc}}(\mathbb{R}^3 \backslash \{0\} \backslash \Lambda)$ such that
\begin{align*}
\left\{ \begin{aligned} - \varepsilon^2 ( \Delta + H) \Psi_{\varepsilon} + (1 - \mu) V \Psi_{\varepsilon} & = 0 \quad & \text{ in } \mathbb{R}^3 \backslash \Lambda, \\
\Psi_{\varepsilon} & = 1 \quad & \text{ on } \partial \Lambda,
\end{aligned} \right.
\end{align*}
and 
\begin{align*}
\int_{\mathbb{R}^3 \backslash \Lambda} | \nabla \Psi_{\varepsilon}|^2 + \frac{|\Psi_{\varepsilon}|^2}{|x|^2} < + \infty.
\end{align*}
We also have the following estimate for every $x \in \mathbb{R}^3 \setminus(\Lambda\cup \{0\})$ and $C > 0$
\begin{align*}
0 < \Psi_{\varepsilon}(x) \leq \frac{C}{1 + |x|}.
\end{align*}
\begin{itemize}
\item[(i)] If we assume in addition that ($V^\infty$) holds with $\alpha = 2$, then, for every $\nu > 1$ and for every $R > 1$, with $\bar{\Lambda} \subset
B(0,R)$, there exist $C>0$ and $\varepsilon_0 > 0$ such that, for every $\varepsilon < \varepsilon_0$ and for every $x \in \mathbb{R}^3 \backslash B(0,R)$,
\begin{align*}
0 < \Psi_{\varepsilon}(x) \leq \frac{C}{|x|^\nu};
\end{align*}
\item[(ii)] If we assume that ($V^\infty$) holds with $\alpha < 2$, then, for every $\nu >0$ and for every $R > 1$, with $\bar{\Lambda} \subset B(0,R)$, there exist $C> 0$ and $\varepsilon_0 > 0$ such that, for every $\varepsilon < \varepsilon_0$ and for every $x \in \mathbb{R}^3 \backslash B(0,R)$,
\begin{align*} 
0 < \Psi_{\varepsilon}(x) \leq C \exp\left( - \nu |x|^{\frac{2-\alpha}{2}} \right);
\end{align*}
\item[(iii)]If we assume that ($V^0$) holds with $\alpha = 2$, then, for every $\nu >0$ and for every $0 < r < 1$, with $B(0,r) \cap \bar{\Lambda} = \emptyset$, there exist $C > 0$ and $\varepsilon_0 > 0$ such that, for every $\varepsilon < \varepsilon_0$ and for every $x \in B(0,r)\setminus\{0\}$,
\begin{align*}
0 < \Psi_{\varepsilon}(x) \leq C |x|^\nu.
\end{align*}
\item[(iv)] If we assume that ($V^0$) holds with $\alpha > 2$, then, for every $\nu >0$ and for every $0 < r < 1$, with $B(0,r) \cap \bar{\Lambda} = \emptyset$, there exist $C > 0$ and $\varepsilon_0 >0$ such that, for every $\varepsilon < \varepsilon_0$ and for every $x \in B(0,r)\setminus\{0\}$,
\begin{align*}
0 < \Psi_{\varepsilon}(x) \leq C \exp \left( - \nu |x|^{\frac{2-\alpha}{2}} \right).
\end{align*}
\end{itemize}
\end{lemma}

We refer to \cite{B-DC-VS} for the proof. 

Now, we construct a comparison function inside of $\Lambda$.
\begin{lemma} \label{lemma:Phi_in_Lambda}
Consider $r>0$. Let $(x_\varepsilon)_\varepsilon = (\rho_\varepsilon \cos \theta, \rho_\varepsilon \sin \theta, x_{3,\varepsilon})_\varepsilon \subset \Lambda$, $\theta \in [0, 2 \pi)$, and $R >0$ be such that $B_{cyl}(x_\varepsilon, R) \subset \Lambda$. We define
\begin{align*} 
\Phi_{\varepsilon}(x) = \cosh \left( \lambda \frac{R - \mathrm{d}_{cyl}(x,x_\varepsilon)}{\varepsilon} \right),
\end{align*}
where $\lambda > 0$ is chosen such that
\begin{align*}
\inf_{\bar{\Lambda}}V > \frac{\lambda^2}{(1-\mu)}.
\end{align*}
Then, there exists $\varepsilon_0 > 0$ such that, for every $\varepsilon < \varepsilon_0$,
\begin{align*}
- \varepsilon^2 \left( \Delta + H \right) \Phi_{\varepsilon} + (1 - \mu) V \Phi_{\varepsilon} \geq 0 \quad \text{ in } B_{cyl}(x_\varepsilon,R) \backslash B_{cyl}(x_\varepsilon,\varepsilon r).
\end{align*}
\end{lemma}

\begin{proof}
By simple calculation, we obtain that
\begin{align*}
& - \varepsilon^2 \left( \Delta + H \right) \Phi_{\varepsilon} + ( 1 - \mu) V \Phi_{\varepsilon} =  \\
& ( \lambda^2 + (1-\mu)V ) \, \Phi_{\varepsilon} - \varepsilon^2 H \Phi_{\varepsilon} + \varepsilon \lambda \frac{2\rho - \rho_\varepsilon}{\rho \, \mathrm{d}_{cyl}(x,x_\varepsilon)}\sinh \left(\lambda \frac{R - \mathrm{d}_{cyl}(x,x_\varepsilon)}{\varepsilon}\right) \geq 0, 
\end{align*}
thanks to the assumption on $\lambda$ and for $\varepsilon$ small enough.
\end{proof}

Thanks to Proposition \ref{prop:uniform_convergence_outside_small_balls} we remark that the assumption $B_{cyl}(x_\varepsilon, R) \subset \Lambda$ is verified if  $\varepsilon$ is taken sufficiently small. From now, we will always consider that $\varepsilon_0 \in \mathbb{R}$ is taken small enough to have this property.

With those two functions $\Psi_{\varepsilon}$ and $\Phi_{\varepsilon}$, we are ready to construct the barrier functions. Again we refer to \cite{B-DC-VS} for the proof. 

\begin{lemma} \label{lemma:barrier_function}
Take $r>r_0$ ($r_0$ introduced in Lemma \ref{lemma:inequation_outside_small_balls}). Let $\lambda > 0$ be as in Lemma \ref{lemma:Phi_in_Lambda} and $(x_\varepsilon)_\varepsilon$ be as in Proposition \ref{prop:uniform_convergence_outside_small_balls}.
Then, there exists $\varepsilon_0 >0$ and a family $(w_{\varepsilon})_\varepsilon \subset C^{1,\alpha}_{\text{loc}}(\mathbb{R}^3 \backslash \{0\} \backslash B_{cyl}(x_\varepsilon, \varepsilon r) )$ of barrier functions such that for $\varepsilon < \varepsilon_0$
\begin{align*}
0 < w_{\varepsilon}(x) \leq C \exp \left( - \frac{\lambda}{\varepsilon} \frac{\mathrm{d}_{cyl}(x,x_\varepsilon)}{1 + \mathrm{d}_{cyl}(x,x_\varepsilon)} \right) ( 1 + |x|)^{-1} \quad \forall x \in \mathbb{R}^3 \setminus(B_{cyl}(x_\varepsilon, \varepsilon r)\cup \{0\}).
\end{align*}
Moreover, if we assume that
\begin{itemize}
\item[(i)] ($V^\infty$) holds with $\alpha = 2$, then, for every $\nu > 1$ and for every $R >1$ with $\bar{\Lambda} \subset B(0,R)$, there exist $C>0$ and $\varepsilon_0$ (eventually smaller than the previous one) such that, for all $\varepsilon < \varepsilon_0$,
\begin{align} \label{eq:bound_barrier_function_V_2_infty}
0 < w_{\varepsilon}(x) \leq  C \exp \left( - \frac{\lambda}{\varepsilon} \frac{\mathrm{d}_{cyl}(x,x_\varepsilon)}{1 + \mathrm{d}_{cyl}(x,x_\varepsilon)} \right) |x|^{-\nu} \quad \forall x \in \mathbb{R}^3 \backslash B(0,R);
\end{align}
\item[(ii)] ($V^\infty$) holds with $\alpha < 2$, then, for every $\nu > 1$ and for every $R >1$ with $\bar{\Lambda} \subset B(0,R)$, there exist $C > 0$ and $\varepsilon_0 > 0$ such that, for all $\varepsilon < \varepsilon_0$,
\begin{align} \label{eq:bound_barrier_function_V_3_infty}
0 < w_{\varepsilon}(x) \leq  C \exp \left( - \frac{\lambda}{\varepsilon} \frac{\mathrm{d}_{cyl}(x,x_\varepsilon)}{1 + \mathrm{d}_{cyl}(x,x_\varepsilon)} \right) \exp \left( - \nu |x|^{\frac{2-\alpha}{2}} \right) \quad \forall x \in \mathbb{R}^3 \backslash B(0,R);
\end{align}
\item[(iii)] ($V^0$) holds with $\alpha = 2$, then, for every $\nu > 1$ and for every $r <1$ with $B(0,r) \cap \bar{\Lambda} = \emptyset$, there exist $C>0$ and $\varepsilon_0 > 0$ such that, for all $\varepsilon < \varepsilon_0$,
\begin{align} \label{eq:bound_barrier_function_V_2_0}
0 < w_{\varepsilon}(x) \leq  C \exp \left( - \frac{\lambda}{\varepsilon} \frac{\mathrm{d}_{cyl}(x,x_\varepsilon)}{1 + \mathrm{d}_{cyl}(x,x_\varepsilon)} \right) |x|^\nu \quad \forall x \in B(0,r)\setminus\{0\};
\end{align}
\item[(iv)] ($V^0$) holds with $\alpha > 2$, then, for every $\nu > 1$ and for every $r >1$ with $B(0,r) \cap \bar{\Lambda} = \emptyset$, there exist $C >0$ and $\varepsilon_0 > 0$ such that, for all $\varepsilon < \varepsilon_0$,
\begin{align} \label{eq:bound_barrier_function_V_3_0}
0 < w_{\varepsilon}(x) \leq  C \exp \left( - \frac{\lambda}{\varepsilon} \frac{\mathrm{d}_{cyl}(x,x_\varepsilon)}{1 + \mathrm{d}_{cyl}(x,x_\varepsilon)} \right) \exp \left( - \nu |x|^{\frac{2-\alpha}{2}} \right) \quad \forall x \in B(0,r)\setminus\{0\}.
\end{align}
\end{itemize}
\end{lemma}

\subsection{Back to the original equation}

Thanks to Lemmas \ref{lemma:inequation_outside_small_balls} and \ref{lemma:barrier_function}, we obtain an upper bound on $|u_{\varepsilon}|$.

\begin{proposition} \label{prop:bound_on_u}
Suppose the assumptions of Theorem \ref{theorem:existence_solutions_penalized} and Proposition \ref{prop:uniform_convergence_outside_small_balls} are satisfied. Let $\lambda > 0$ be as in Lemma \ref{lemma:Phi_in_Lambda}, $(x_\varepsilon)_\varepsilon \subset \bar{\Lambda}$ be as in Proposition \ref{prop:uniform_convergence_outside_small_balls} and $(u_{\varepsilon})_\varepsilon \subset \mathcal{X}_\varepsilon$ be the solutions found in Theorem \ref{theorem:existence_solutions_penalized}. Then, there exists $C > 0$ and $\varepsilon_0 > 0$ such that, for all $\varepsilon < \varepsilon_0$,
\begin{align} \label{eq:estimate_modulus_solutions}
0 < |u_{\varepsilon}(x)| \leq C \exp \left( - \frac{\lambda}{\varepsilon} \frac{\mathrm{d}_{cyl}(x,x_\varepsilon)}{1 + \mathrm{d}_{cyl}(x,x_\varepsilon)} \right) (1 + |x|)^{-1} \quad \forall x \in \mathbb{R}^3\setminus\{0\}.
\end{align}
Moreover, \eqref{eq:bound_barrier_function_V_2_infty}-\eqref{eq:bound_barrier_function_V_3_0} hold for $|u_{\varepsilon}|$ in place of $w_{\varepsilon}$ if we make the same assumptions on $V$.
\end{proposition}

\begin{proof}
By Lemma \ref{lemma:inequation_outside_small_balls}, we know that $|u_{\varepsilon}|$ is a subsolution in $\mathbb{R}^3 \backslash B_{cyl}(x_\varepsilon, \varepsilon r)$, for some $r > r_0$. Furthermore, thanks to Lemma \ref{lemma:convergence_rescaled_solutions}, we know that $\| u_{\varepsilon} \|_{L^{\infty}(B_{cyl}(x_\varepsilon,\varepsilon r))}$ is bounded for $\varepsilon < \varepsilon_0$. We deduce that $|u_{\varepsilon}| \leq \| u_{\varepsilon} \|_{L^{\infty}(B_{cyl}(x_\varepsilon,\varepsilon r))}$ on $\partial B_{cyl}(x_\varepsilon, \varepsilon r)$. With the comparison principle, we conclude that
\begin{align*}
0 < |u_{\varepsilon}(x) | \leq \| u_{\varepsilon} \|_{L^{\infty}(B_{cyl}(x_\varepsilon, \varepsilon r))} w_{\varepsilon}(x) \quad \forall x \in \mathbb{R}^3  \setminus(B_{cyl}(x_\varepsilon, \varepsilon r)\cup \{0\}).
\end{align*}
Finally, since $|u_{\varepsilon}|$ is bounded in $B_{cyl}(x_\varepsilon, \varepsilon r)$, we obtain the estimate \eqref{eq:estimate_modulus_solutions} for all $x \in \mathbb{R}^3$. The other estimates follow by reasoning in the same way.
\end{proof}

We can now proof the main Theorem.

\begin{proof}[Proof of Theorem \ref{theorem:main}.]
It remains us to prove that $u_{\varepsilon}$ is in fact a solution of the initial problem \eqref{eq:initialproblem}. For this, we need to show that
\begin{align*}
f(|u_{\varepsilon}|^2) = |u_{\varepsilon}|^{p-2} \leq \varepsilon^2 H(x) + \mu V(x) \quad   \forall x \in \mathbb{R}^3 \backslash \Lambda.
\end{align*}
We prove this for example in the case where we make no assumptions on $V$ (then $p> 4$). We use Proposition \ref{prop:bound_on_u} to say that
\begin{align*}
|u_{\varepsilon}(x)|^{p-2} \leq C e^{-\frac{\lambda}{\varepsilon}(p-2)} (1 + |x|)^{-(p-2)} \leq \varepsilon^2 H(x) + \mu V(x),
\end{align*}
for small $\varepsilon$. The last inequality is verified since we considered $p>4$. Indeed, for $|x|$ large, the right-hand side behaves as $1/\left( |x|^2 \log |x| \right)$. The left-hand side decays then faster since it behaves as $1/|x|^{p-2}$. For $|x|$ small, the left-hand side behaves as a constant while the right-hand side is unbounded. The other cases may be treated in a similar way. 
\end{proof}

\begin{remark}
In addition of theorem \ref{theorem:main}, we may also prove that estimates \eqref{eq:bound_barrier_function_V_2_infty}-\eqref{eq:bound_barrier_function_V_3_0} hold for $|u_{\varepsilon}|$ instead of $w_{\varepsilon}$ if we make the corresponding assumptions on $V$. 
\end{remark}

\section{Another class of symmetric solutions} \label{section:another-class-symmetric-solutions}

When $A$ is equal to the Lorentz potential, i.e. $A= (-x_2,x_1,0)$ or has the slightly more general form 
\begin{align}\label{AtypeLo}
A(\rho, \theta, x_3) = c(\rho)(-\sin \theta, \cos \theta,0),
\end{align}
Esteban and Lions have proposed in \cite[Section 4.3]{EsLi} the class of solutions
\begin{align*}
u_k := C_k\left(\frac{x_2+ix_1}{\rho}\right)^kv_k,
\end{align*} 
where $k\in\Z$, $C_{k}\in\R\setminus\{0\}$ and $v_k$ are real cylindrically symmetric solutions of an auxiliary problem. One can check easily that the functions $u_k$ solve 
\begin{equation} \label{eq:original}
	\left(i \varepsilon \nabla + A \right)^2 u_k + V(\rho,x_3) u_k = |u_k|^{p-2} u_k, \qquad x\in \mathbb{R}^{3},
\end{equation}
if and only if the $v_k$ are real solutions of 
\begin{equation} \label{eq:kvortices}
	- \varepsilon ^2 \Delta v_k + \left(\left(\frac{k\varepsilon}{\rho} + c(\rho)\right)^2 + V(\rho,x_3) \right) v_k = C_{k}^{p-2}|v_k|^{p-2} v_k, \qquad x\in \mathbb{R}^{3}.
\end{equation}
The limit equation in $\mathbb{R}^2$ has the form
\begin{align}
- \Delta w_k + \left( c^2(\rho_0) + V(\rho_0, x_{3,0}) \right) w_k = C_k^{p-2} |w_k|^{p-2} w_{k},
\end{align}
where $(\rho_0, x_{3,0})$ is such that the normalized concentration function
\begin{align}\label{newM}
\mathcal{M}(\rho, x_3) = \rho \left( c^2(\rho) + V(\rho, x_3) \right)^{\frac{2}{p-2}}
\end{align}
is locally minimized at this point. 

Observe that this reduction to a real valued problem allows us to use directly the arguments from \cite{B-DC-VS} without much modifications. One can then consider several cases according to the properties of $c$ and $V$. We do not address all these cases in details. We will focus on the special case which for instance allows to consider a critical frequency. 

\subsection{Existence at the critical frequency}

Remember that the potential $V$ stands for $U-E$, where $U$ is the electrical potential and $E$ is the frequency of the standing wave $\psi(x,t)= e^{-i \frac{E}{\hbar} t}  \, u(x)$. When $E=\inf_{\R^{N}}U(x)$, we say that $E$ is the critical frequency. When $A=0$, the critical frequency was studied by many authors, starting with the contribution of Byeon and Wang \cite{ByeonWang2002,ByeonWang2003} and followed by many others.

Byeon and Wang have shown that there exists a standing wave which is trapped in a neighbourhood of the isolated minimum points of $V$ and whose amplitude goes to $0$ as $\hbar\to 0$. Moreover, depending upon the local behaviour of the potential function $V$ near the minimum points, the limiting profile of the standing-wave solutions was shown to exhibit quite different characteristic features. This is in striking contrast with the non-critical frequency case ($\inf U(x)>E$) where the solution develops a spike in the semiclassical limit. 

Here we show that even if the frequency is critical, the presence of an external magnetic field allows for the existence of a solution concentrating on a circle and whose amplitude does not vanish in the semiclassical limit so that this solution is a spike type solution.

\medbreak
%

Let $p>2$ and $k\in\Z$. Let $V \in C(\mathbb{R}^3\backslash \{0\})$ be nonnegative and such that $V (g x) = V(x)$ for every $g \in G$. Assume $A \in C^{1}(\mathbb{R}^3, \mathbb{R}^3)$ is of the form \eqref{AtypeLo} and such that $c(\rho)>0$ for every $\rho$ and
$$\liminf_{\rho\to\infty}c(\rho)\rho>0.$$

With those assumptions, the assumption $\liminf_{|x| \to + \infty} W(x) |x|^2 > 0$ holds for the potential 
\begin{align*}
W = \left(\frac{k\varepsilon}{\rho} + c(\rho)\right)^2 + V(\rho,x_3)
\end{align*}
and $\varepsilon$ small. Moreover this potential is nonnegative everywhere and
for every $0<\theta < 1$ and $\varepsilon>0$, there exists $\rho_{\theta,\varepsilon}>0$ such that 
$$\left(\frac{k\varepsilon}{\rho} + c(\rho)\right)^2 + V(\rho,x_3) \ge \theta c(\rho)^2 + V(\rho,x_3),$$
for $\rho\ge \rho_{\theta,\varepsilon}$. Clearly $\rho_{\theta,\varepsilon}\to 0$ as $\varepsilon\to 0$ for any fixed $\theta$.

The proof of the following theorem can be easily recovered from \cite{B-DC-VS} with straightforward modifications. 

\begin{theorem}
With the above conditions on $c$, $V$, $k$ and $p$, assume there exists a bounded $G$-invariant smooth set $\Lambda \subset \mathbb{R}^3$ such that \eqref{eq:condition_sur_Lambda} is satisfied with $\mathcal M$ now defined by \eqref{newM} and \eqref{eq:condition_sur_Lambda2} is satisfied for $\theta c^2 + V$, $\theta \in (0,1)$. 
If $\varepsilon>0$ is small enough, the equation \eqref{eq:kvortices} has a solution $v_{k,\varepsilon}$ such that $v_{k,\varepsilon}(gx) = v_{k,\varepsilon}(x)$ for all $g\in G$, 
 $v_{k,\varepsilon}$ attains its maximum at some $x_{k,\varepsilon}=(\rho_{k,\varepsilon} \cos \theta, \rho_{k,\varepsilon} \sin \theta, x_{3,k,\varepsilon}) \in \Lambda$ such that 
\begin{itemize}
\item[(ii)]
$
\displaystyle \liminf_{\varepsilon \to 0} |v_{k,\varepsilon}(x_{k,\varepsilon})| > 0 ;
$
\item[(iii)]$ \displaystyle \lim_{\varepsilon \to 0} \mathcal{M}(x_{k,\varepsilon}) = \inf_{\Lambda \cap \mathcal{H}^\perp} \mathcal{M}$;
\item[(iv)] $\displaystyle \limsup_{\varepsilon \to 0} \frac{\mathrm{d}_{cyl}(x_{k,\varepsilon}, \mathcal{H}^\perp)}{\varepsilon} < + \infty$  , that is $x_{3,k,\varepsilon} \to 0$; \\
\item[(v)] $ \displaystyle \liminf_{\varepsilon \to 0} \mathrm{d}_{cyl}(x_{k,\varepsilon}, \partial \Lambda)> 0$.
\end{itemize}
Finally, there exists $C_{k}\in\R\setminus \{0\}$ such that 
$$u_{k,\varepsilon} = C_k\left(\frac{x_2+ix_1}{\rho}\right)^kv_{k,\varepsilon}$$
solves \eqref{eq:original} and
for every $\nu > 1$, the asymptotic estimate
$$
\displaystyle 0 < |u_{k,\varepsilon}(x)| \leq C\exp \left( - \frac{\lambda}{\varepsilon} \frac{\mathrm{d}_{cyl}(x,x_{k,\varepsilon})}{1 + \mathrm{d}_{cyl}(x,x_{k,\varepsilon})} \right) |x|^{-\nu} \quad \forall x \in \mathbb{R}^3\setminus\{0\}.
$$
%

\end{theorem}

As previously discussed, we can consider the critical frequency as one can allow $V$ to vanish at the local minimum point of  
$\mathcal M$ in $\Lambda$. 

Observe also that the ansatz fixes the concentration as the concentration set is the same for any choice of $k\in\Z$.
%
%

\bibliographystyle{plain}
\bibliography{biblio-BonCinNys.bib}

\begin{thebibliography}{10}

\bibitem{ABC1997}
A.~Ambrosetti, M.~Badiale, and S.~Cingolani.
\newblock Semiclassical states of nonlinear {S}chr\"odinger equations.
\newblock {\em Arch. Rational Mech. Anal.}, 140(3):285--300, 1997.

\bibitem{AmbrosettiRuiz2006}
A.~Ambrosetti and D.~Ruiz.
\newblock Radial solutions concentrating on spheres of nonlinear
  {S}chr\"odinger equations with vanishing potentials.
\newblock {\em Proc. Roy. Soc. Edinburgh Sect. A}, 136(5):889--907, 2006.

\bibitem{AmbrosettiMalchiodiBook}
Antonio Ambrosetti and Andrea Malchiodi.
\newblock {\em Perturbation methods and semilinear elliptic problems on {${\bf
  R}^n$}}, volume 240 of {\em Progress in Mathematics}.
\newblock Birkh\"auser Verlag, Basel, 2006.

\bibitem{AmbrosettiMalchiodiNi2003}
Antonio Ambrosetti, Andrea Malchiodi, and Wei-Ming Ni.
\newblock Singularly perturbed elliptic equations with symmetry: existence of
  solutions concentrating on spheres. {I}.
\newblock {\em Comm. Math. Phys.}, 235(3):427--466, 2003.

\bibitem{ArioliSzulkin}
Gianni Arioli and Andrzej Szulkin.
\newblock A semilinear {S}chr\"odinger equation in the presence of a magnetic
  field.
\newblock {\em Arch. Ration. Mech. Anal.}, 170(4):277--295, 2003.

\bibitem{BadialeDaprile2002}
Marino Badiale and Teresa D'Aprile.
\newblock Concentration around a sphere for a singularly perturbed
  {S}chr\"odinger equation.
\newblock {\em Nonlinear Anal.}, 49(7, Ser. A: Theory Methods):947--985, 2002.

\bibitem{B-DC-VS}
Denis Bonheure, Jonathan Di~Cosmo, and Jean Van~Schaftingen.
\newblock Nonlinear {S}chr\"odinger equation with unbounded or vanishing
  potentials: solutions concentrating on lower dimensional spheres.
\newblock {\em J. Differential Equations}, 252(2):941--968, 2012.

\bibitem{BonheureVanschaftingen}
Denis Bonheure and Jean Van~Schaftingen.
\newblock Bound state solutions for a class of nonlinear {S}chr\"odinger
  equations.
\newblock {\em Rev. Mat. Iberoam.}, 24(1):297--351, 2008.

\bibitem{ByeonWang2002}
Jaeyoung Byeon and Zhi-Qiang Wang.
\newblock Standing waves with a critical frequency for nonlinear
  {S}chr\"odinger equations.
\newblock {\em Arch. Ration. Mech. Anal.}, 165(4):295--316, 2002.

\bibitem{ByeonWang2003}
Jaeyoung Byeon and Zhi-Qiang Wang.
\newblock Standing waves with a critical frequency for nonlinear
  {S}chr\"odinger equations. {II}.
\newblock {\em Calc. Var. Partial Differential Equations}, 18(2):207--219,
  2003.

\bibitem{Cingolani2003}
Silvia Cingolani.
\newblock Semiclassical stationary states of nonlinear {S}chr\"odinger
  equations with an external magnetic field.
\newblock {\em J. Differential Equations}, 188(1):52--79, 2003.

\bibitem{CingolaniClapp2009}
Silvia Cingolani and M{\'o}nica Clapp.
\newblock Intertwining semiclassical bound states to a nonlinear magnetic
  {S}chr\"odinger equation.
\newblock {\em Nonlinearity}, 22(9):2309--2331, 2009.

\bibitem{CingolaniJeanjeanSecchi2009}
Silvia Cingolani, Louis Jeanjean, and Simone Secchi.
\newblock Multi-peak solutions for magnetic {NLS} equations without
  non-degeneracy conditions.
\newblock {\em ESAIM Control Optim. Calc. Var.}, 15(3):653--675, 2009.

\bibitem{CingolaniLazzo1997}
Silvia Cingolani and Monica Lazzo.
\newblock Multiple semiclassical standing waves for a class of nonlinear
  {S}chr\"odinger equations.
\newblock {\em Topol. Methods Nonlinear Anal.}, 10(1):1--13, 1997.

\bibitem{CingolaniLazzo2000}
Silvia Cingolani and Monica Lazzo.
\newblock Multiple positive solutions to nonlinear {S}chr\"odinger equations
  with competing potential functions.
\newblock {\em J. Differential Equations}, 160(1):118--138, 2000.

\bibitem{CingolaniSecchi2002}
Silvia Cingolani and Simone Secchi.
\newblock Semiclassical limit for nonlinear {S}chr\"odinger equations with
  electromagnetic fields.
\newblock {\em J. Math. Anal. Appl.}, 275(1):108--130, 2002.

\bibitem{CingolaniSecchi2005}
Silvia Cingolani and Simone Secchi.
\newblock Semiclassical states for {NLS} equations with magnetic potentials
  having polynomial growths.
\newblock {\em J. Math. Phys.}, 46(5):053503, 19, 2005.

\bibitem{ClappSzulkin}
M{\'o}nica Clapp and Andrzej Szulkin.
\newblock Multiple solutions to nonlinear {S}chr\"odinger equations with
  singular electromagnetic potential.
\newblock {\em J. Fixed Point Theory Appl.}, 13(1):85--102, 2013.

\bibitem{DP-F}
Manuel del Pino and Patricio~L. Felmer.
\newblock Local mountain passes for semilinear elliptic problems in unbounded
  domains.
\newblock {\em Calc. Var. Partial Differential Equations}, 4(2):121--137, 1996.

\bibitem{DelPinoFelmer1997}
Manuel del Pino and Patricio~L. Felmer.
\newblock Semi-classical states for nonlinear {S}chr\"odinger equations.
\newblock {\em J. Funct. Anal.}, 149(1):245--265, 1997.

\bibitem{DelPinoKowalczykWei2007}
Manuel del Pino, Michal Kowalczyk, and Jun-Cheng Wei.
\newblock Concentration on curves for nonlinear {S}chr\"odinger equations.
\newblock {\em Comm. Pure Appl. Math.}, 60(1):113--146, 2007.

\bibitem{EsLi}
Maria~J. Esteban and Pierre-Louis Lions.
\newblock Stationary solutions of nonlinear {S}chr\"odinger equations with an
  external magnetic field.
\newblock In {\em Partial differential equations and the calculus of
  variations, {V}ol.\ {I}}, volume~1 of {\em Progr. Nonlinear Differential
  Equations Appl.}, pages 401--449. Birkh\"auser Boston, Boston, MA, 1989.

\bibitem{FW1986}
Andreas Floer and Alan Weinstein.
\newblock Nonspreading wave packets for the cubic {S}chr\"odinger equation with
  a bounded potential.
\newblock {\em J. Funct. Anal.}, 69(3):397--408, 1986.

\bibitem{GilbargTrudinger}
David Gilbarg and Neil~S. Trudinger.
\newblock {\em Elliptic partial differential equations of second order}, volume
  224 of {\em Grundlehren der Mathematischen Wissenschaften [Fundamental
  Principles of Mathematical Sciences]}.
\newblock Springer-Verlag, Berlin, second edition, 1983.

\bibitem{Kurata}
Kazuhiro Kurata.
\newblock Existence and semi-classical limit of the least energy solution to a
  nonlinear {S}chr\"odinger equation with electromagnetic fields.
\newblock {\em Nonlinear Anal.}, 41(5-6, Ser. A: Theory Methods):763--778,
  2000.

\bibitem{Malchiodi2005}
A.~Malchiodi.
\newblock Concentration at curves for a singularly perturbed {N}eumann problem
  in three-dimensional domains.
\newblock {\em Geom. Funct. Anal.}, 15(6):1162--1222, 2005.

\bibitem{Malchiodi2004}
Andrea Malchiodi.
\newblock Solutions concentrating at curves for some singularly perturbed
  elliptic problems.
\newblock {\em C. R. Math. Acad. Sci. Paris}, 338(10):775--780, 2004.

\bibitem{MalchiodiMontenegro2002}
Andrea Malchiodi and Marcelo Montenegro.
\newblock Boundary concentration phenomena for a singularly perturbed elliptic
  problem.
\newblock {\em Comm. Pure Appl. Math.}, 55(12):1507--1568, 2002.

\bibitem{MollePassaseo2006}
Riccardo Molle and Donato Passaseo.
\newblock Concentration phenomena for solutions of superlinear elliptic
  problems.
\newblock {\em Ann. Inst. H. Poincar\'e Anal. Non Lin\'eaire}, 23(1):63--84,
  2006.

\bibitem{MorozVanschaftingen2009-2}
Vitaly Moroz and Jean Van~Schaftingen.
\newblock Existence and concentration for nonlinear {S}chr\"odinger equations
  with fast decaying potentials.
\newblock {\em C. R. Math. Acad. Sci. Paris}, 347(15-16):921--926, 2009.

\bibitem{MorozVanschaftingen2009-1}
Vitaly Moroz and Jean Van~Schaftingen.
\newblock Semiclassical stationary states for nonlinear {S}chr\"odinger
  equations with fast decaying potentials.
\newblock {\em Calc. Var. Partial Differential Equations}, 37(1-2):1--27, 2010.

\bibitem{Oh1989}
Yong-Geun Oh.
\newblock Existence of semiclassical bound states of nonlinear {S}chr\"odinger
  equations with potentials of the class {$(V)_a$}.
\newblock {\em Comm. Partial Differential Equations}, 13(12):1499--1519, 1988.

\bibitem{Palais}
Richard~S. Palais.
\newblock The principle of symmetric criticality.
\newblock {\em Comm. Math. Phys.}, 69(1):19--30, 1979.

\bibitem{Rabinowitz}
Paul~H. Rabinowitz.
\newblock On a class of nonlinear {S}chr\"odinger equations.
\newblock {\em Z. Angew. Math. Phys.}, 43(2):270--291, 1992.

\bibitem{SecchiSquassina2005}
Simone Secchi and Marco Squassina.
\newblock On the location of spikes for the {S}chr\"odinger equation with
  electromagnetic field.
\newblock {\em Commun. Contemp. Math.}, 7(2):251--268, 2005.

\end{thebibliography}

\end{document}